\documentclass[11pt]{amsart}
\usepackage{amssymb,amsfonts,amsmath,enumerate, psfrag, a4wide}

\usepackage[T1]{fontenc} 
\usepackage[latin1]{inputenc} 

\def\C{\mathbf{C}}
\def\R{\mathbf{R}}
\def\Q{\mathbf{Q}}
\def\Z{\mathbf{Z}}
\def\N{\mathbf{N}}
\def\P{\mathbf{P}}

\def\1{\mathbf{1}}

\def\a{\alpha}

\def\e{\varepsilon}
\def\l{\lambda_1}
\def\f{\varphi}
\def\g{\gamma}
\def\p{\psi}
\def\r{\varrho}
\def\om{\omega}

\def\E{{\mathcal E}}

\def\W{{\mathcal W}}

\def\dist{\mathrm{dist}}

\newcommand{\cc}{\mathbf{C}}

\newcommand{\cv}{\rightarrow}

\newcommand{\set}[1]{\left\{#1\right\}}
\newcommand{\norm}[1]{\left\Vert#1\right\Vert}
\newcommand{\abs}[1]{\left\vert#1\right\vert}
\newcommand{\cd}{\cc^2}

\newcommand{\rest}[1]{ \arrowvert_{#1}}
\newcommand{\m}{{\bf M}}

\newcommand{\qq}{\mathcal{Q}}

\newcommand{\geom}{\dot\wedge}

\newcommand{\pair}[2]{\left\langle #1,#2 \right\rangle}
\newcommand{\eqdef}{:=}
\newcommand{\ignore}[1]{}

\DeclareMathOperator{\supp}{Supp}

\newtheorem{lem}{Lemma}[section]
\newtheorem{pro}[lem]{Proposition}
\newtheorem{defi}[lem]{Definition}
\newtheorem{def/not}[lem]{Definition/Notation}
\newtheorem{Pro/def}[lem]{Proposition/Definition}
\newtheorem{theo}{Theorem}

\newtheorem{thm}[lem]{Theorem}
\newtheorem{cor}[lem]{Corollary}

\theoremstyle{remark}
\newtheorem{rqe}[lem]{Remark}

\newtheorem{exa}[lem]{Example}

\begin{document}

\title[Energy and invariant measure]{Dynamics of meromorphic mappings with small topological degree II: 
Energy and invariant measure}

\author{Jeffrey Diller \and Romain Dujardin \and Vincent Guedj}

\address{Department of Mathematics\\
         University of Notre Dame\\
         Notre Dame, IN 46556}
\email{diller.1@nd.edu}
\address{CMLS\\ \'Ecole Polytechnique \\ 91128 Palaiseau \\
         France}
\email{dujardin@math.polytechnique.fr}
\address{Universit\'e Aix-Marseille 1 \\ LATP \\ 13453 Marseille Cedex 13 \\  France}
\email{guedj@cmi.univ-mrs.fr}


\thanks{J.D. acknowldeges the National Science Foundation for its support through grant DMS 06-53678.}
\thanks{R.D. was partially supported by ANR through project BERKO}

\maketitle

\begin{abstract}
We continue our study of the dynamics of meromorphic mappings with small topological degree $\lambda_2(f)<\lambda_1(f)$ on a compact K\"ahler surface $X$.  Under  general hypotheses we are able to construct a canonical invariant measure which is mixing, does not charge pluripolar sets and has a natural geometric description.

Our hypotheses are always satisfied when $X$ has Kodaira dimension zero, or when the mapping is induced by a polynomial endomorphism of $\C^2$.  They are new even in the birational case ($\lambda_2(f)=1$).  We also exhibit families of mappings where our assumptions are generically satisfied and show that if counterexamples exist, the corresponding measure must give mass to a pluripolar set. 
\end{abstract}

\section*{Introduction}

In this article we continue from \cite{ddg1} our study of the dynamics of meromorphic mappings with small topological degree on complex surfaces.  Whereas our previous article focused on constructing canonical forward and backward invariant currents for a given mapping, here we take up the problem of intersecting these currents to create a natural invariant measure.  In fact, the first half of this article does not concern dynamics at all but rather the general problem of defining the wedge product of two positive closed $(1,1)$ currents on a compact complex surface.  This is of interest in its own right, and our treatment draws much in content and spirit from recent work of Guedj and Zeriahi \cite{gz2} concerning the definition of the complex Monge-Amp\`ere operator in the compact setting.  Let us begin, nevertheless,
by rehearsing the dynamical setting of immediate concern. We refer the reader to \cite{ddg1} for a more thorough presentation and list of references.

\medskip

Let $(X,\omega)$ be a compact K\"ahler surface and $f:X\to X$ a meromorphic self-map.  We assume that $f$ is {\em 1-stable}, meaning that the induced action $f^*:H^{1,1}(X)\to H^{1,1}(X)$ satisfies 
$$
(f^n)^* = (f^*)^n 
\text{ for all } n\in\N. 
$$
We let $\lambda_2 = \lambda_2(f)$ denote the topological degree of $f$ and $\lambda_1 = \lambda_1(f)$ be the first dynamical degree of $f$.  These are the spectral radii of the actions $f^*:H^{j,j}(X)\to H^{j,j}(X)$ for $j=2$ and $j=1$, respectively.  Our main assumption is that $f$ has {\em small topological degree}; that is,
$$
\lambda_2(f)< \lambda_1(f).
$$
Notice that with this terminology a birational map ($\lambda_2 = 1$) 
has small topological degree only when $\lambda_1>1$. Note also that in this case either
$X$ is rational or ${\rm kod}(X)=0$.

Under these conditions we have shown in \cite{ddg1} that there exist closed positive (1,1) currents $T^+$ and $T^-$ uniquely characterized by
\begin{equation}
 \label{invcurrents}
T^+=\lim_{n \rightarrow \infty} \frac{1}{\lambda_1^n} {(f^{n})^* \omega}, \text{ and }
T^-=\lim_{n \rightarrow \infty} \frac{1}{\lambda_1^n}{(f^n)_* \omega}
\end{equation}
Furthermore, $T^\pm$ can be written as $T^\pm = \omega^\pm + dd^c G^\pm$, where $\omega^\pm$ are positive closed currents with bounded  potentials. By scaling $\omega$, we normalize the cohomological intersection number so that $\set{T^+}\cdot \set{T^-} = 1$. Finally, under the additional assumption that $X$ is projective (non-projective examples can arise on surfaces of Kodaira dimension zero), the invariant currents have special geometric properties: $T^+$ is laminar while $T^-$ is woven (see \cite{ddg1} for definitions).

\medskip
We have three goals in this article.  
\begin{itemize}
 \item[-]
First, we seek general conditions under which the wedge product $T^+\wedge T^-$ may be reasonably defined as a probability measure $\mu_f$ on $X$. Here,  ``reasonably defined'' is understood as a continuity requirement: when $T^\pm$ are approximated in a standard pluripotential theoretic sense by convergent sequences $(T^\pm_j)$ of less singular positive closed $(1,1)$ currents, then we insist that $\lim_{j\to\infty} T^+_j\wedge T^-_j\to \mu_f$.

 \item[-] Second, we seek to show that the conditions that guarantee existence of $\mu_f$ also ensure that it has good dynamical properties: 
\begin{enumerate}[{\em i.}]
 \item $\mu_f$ is invariant under $f$; 
 \item $\mu_f$ is mixing;
 \item $\mu_f$ may be alternately viewed as {\em geometric intersection} of the laminar/woven structures of the currents $T^+$ and $T^-$.
\end{enumerate} 
In a third article, \cite{ddg3} we will greatly elaborate on this second goal, studying the fine ergodic properties of $(f,\mu_f)$ when  properties {\em i.} to  {\em iii.} are satisfied.

 \item[-] Finally we seek to apply our results to particular examples of meromorphic mappings.  
This requires that we find checkable dynamical criteria that imply 
the potential theoretic conditions needed to define $\mu_f$.
\end{itemize}

It should be noted that, as with the other two articles \cite{ddg1,ddg3} in this series, 
this one represents an attempt to generalize things that are known about birational maps to 
the larger setting of maps with small topological degree. On the other hand, this paper is the 
only one in the series that gives new results even in the birational setting.

\medskip
\begin{center}
 $\diamond$
\end{center}
\medskip

Let us now describe our results. Regarding the first goal, let $S = \alpha + dd^c \f$ and $T= \beta + dd^c \p$ 
be any two positive closed $(1,1)$ currents on $X$, each expressed in terms of a 
positive current with bounded local potentials ($\alpha$, $\beta$) and a negative (relative) potential function 
($\f$, $\p$).  
Products $\alpha \wedge T$ involving a positive current with bounded potentials are 
well-understood \cite{bt1}.  
The most straightforward way to define $S\wedge T$ more generally is to require 
$\f \in L^1(T)$ and then declare $dd^c \f \wedge T \eqdef dd^c(\f T)$.  
We write ``$T\in L^1(S)$'' when this integrability holds.  
This condition is independent of the choice of decomposition of $T$ and moreover 
symmetric in $S$ and $T$.  

In our dynamical situation, the condition $T^+\in L^1(T^-)$ suffices for our first goal.  However, it is not clear that the resulting measure $\mu_f$ then has good dynamical properties.  In particular it seems difficult, given only the $L^1$ condition, to establish that the intersection $T^+\wedge T^-$ is geometric.  Moreover, in certain cases (see e.g. the examples in \S \ref{subs:irrational}) we do not know whether $T^+\in L^1(T^-)$, but we are nevertheless able to define a ``reasonable'' wedge product $T^+\wedge T^-$ by taking advantage of the global (compact K\"ahler) context, in the spirit of \cite{gz2}.

Let $S$ and $T$ be positive closed currents as before. We say that {\em $S$ has finite $T$-energy},  and denote ``$S \in \E(T)$'', if there exists an unbounded convex increasing function $\chi:(-\infty,0]\to (-\infty,0]$ such that 
$\chi \circ \f \in L^1(\alpha \wedge T)$ and the ``weighted energy''
\begin{equation}\label{eq:weight}
 \int \chi'\circ \f\,d\f\wedge d^c\f \wedge T
\end{equation}
is finite.  This condition is independent of the choice of decomposition for $S$. We stress that it does not imply that
$S\in L^1(T)$.

Our first main result is the following.

\begin{theo}
\label{theo:intersect} Let $S= \alpha + dd^c \f$ and $T=\beta + dd^c \p$ 
be positive closed currents on a compact K\"ahler surface as above. 
Assume that $S\in \E(T)$, $T \in \E(S)$ and $T$ does not charge pluripolar sets. 

Then the product $S\wedge T$ is well-defined and does not charge pluripolar sets.
\end{theo}

When the first assumption of this theorem holds for the invariant currents 
$T^+$ and $T^-$ associated to a 1-stable meromorphic map $f:X\to X$ of small 
topological degree, we say that $f$ {\em has finite dynamical energy}.
In this case we can show that the current $T^+$ does not charge pluripolar sets (see 
Proposition \ref{pro:alternativeT+}).
Thus the measure $\mu_f=T^+ \wedge T^-$ is well defined and does not charge pluripolar sets.
We are further able to prove that $\mu_f$ has the desired properties (i)-(iii) above, property (iii) being the hardest to establish.

\begin{theo}
\label{theo:properties} 
Let $f$ be a 1-stable meromorphic self map of a compact K\"ahler surface, 
with small topological degree.  Assume that  $f$ has finite dynamical energy. 
Then $\mu_f \eqdef T^+\wedge T^-$
is invariant and mixing.

If furthermore $X$ is projective then  $\mu_f$ is described by the geometric intersection of the laminar/woven structures of $T^{+/-}$.
\end{theo}

\medskip

We now turn to the problem of checking the finite dynamical energy condition; i.e. of finding a weight $\chi$ so that the integral in \eqref{eq:weight} is demonstrably finite.   Some weights are, of course, easier to work with than others.  The class of {\em homogeneous weights} $\chi(t) = -(-t)^p$, $0<p\leq 1$ turns out to be particularly useful for dynamical applications.

\smallskip

For the weight $\chi(t) = t$ we obtain the following criterion which generalizes the work of Bedford and the first author \cite{bd} (see also \cite{dg}).  In order to state it, we let $I^+$ denote the indeterminacy set of $f$, and $I^- = f(E^+)$ denote the image of the exceptional set $E^+$ of $f$.

\begin{theo}
\label{theo:criterion}  
Let $f$ be a 1-stable meromorphic self map of a compact K\"ahler surface, 
with small topological degree, and let $T^\pm = \omega^\pm + dd^cG^\pm$ be the invariant currents.
Suppose that $G^+$ is finite at each point of $I^-$ and $G^-$ is finite at each point 
in $I^+$. Assume furthermore that
\begin{itemize}
\item[-] either $T^+ \in L^1(T^-)$,
\item[-] or no point in either $I^+$ or $I^-$ is spurious.  
\end{itemize}
Then $T^+ \in \E(T^-)$ and 
$T^-\in \E(T^+)$ with weight $\chi(t) = t$.
\end{theo}

We do not define ``spurious'' here (see \S \ref{sec:homogeneous}), but we do note 
that the second assumption is satisfied when, for instance, the classes of 
$T^+$ and $T^-$ are K\"ahler.  Also, when $\lambda_2 =1$ it holds
up to a birational change of surface \cite[Prop. 4.1]{bd}.  
Finiteness of $G^\pm$ on $I^\mp$ is a kind of ``avoidance'' condition on the orbits of 
$I^+$ and $I^-$ that is readily verified for many maps.  

\noindent We show in {\it section 4} that the hypotheses of Theorem \ref{theo:criterion}
are satisfied in particular for

\begin{itemize}
\item[-] polynomial maps of $\C^2$, 
\item[-] maps on $\P^1\times\P^1$ that come from the secant 
algorithm for finding roots of polynomials, 
\item[-] maps on surfaces $X$ with ${\rm kod}(X)=0$.
\end{itemize}

We also give weaker ``indeterminacy avoidance'' conditions that guarantee finite energy with respect to the 
other homogeneous weights $\chi(t)=-(-t)^p$.  By studying a birational example of Favre, whose salient feature is an invariant line on which the map $f$ acts by rotation, we show how it can happen that the weaker avoidance conditions are verified when the stronger ones are not. For birational mappings with $T^+\in L^1(T^-)$ it was proven in \cite{dg} 
that 
$$
\text{ \it finite energy with respect to the weight }
\chi(t) = t \text{ is equivalent to } \log d(\cdot, I^+) \in L^1(\mu_f).
$$
On the other hand, it follows from Theorem \ref{theo:properties} and the results of 
 \cite{birat} that, if   a birational map $f$ with $\lambda_1>1$ has finite dynamical energy then it  enjoys many interesting dynamical properties.  Thus, the examples in \S \ref{subs:irrational} show that for  birational maps, it is possible that the results of \cite{birat} hold while $\log d(\cdot, I^+) \notin L^1(\mu_f)$.  

Less formally but more suggestively, our methods allow us to prove that $\mu_f$ acts much like 
a non-uniformly hyperbolic measure 
{\it even though $\mu_f$ does not satisfy a convenient hypothesis guaranteeing existence of 
well-defined Lyapunov exponents} 
(see e.g. \cite[Section S.2]{KaHa} for background on these).

\smallskip

It is worth emphasizing that 
{\it we do not know any example violating the finite dynamical energy condition}. 
One might hope that no such examples exist.  Regardless, a rational map 
with small topological degree and infinite dynamical energy would likely 
have some very surprising properties.  For instance, we obtain the 
following dichotomy:

\begin{theo}
\label{theo:eitheror}
Let $f$ be a 1-stable meromorphic self map of a compact K\"ahler surface, 
with small topological degree.
Assume further that $T^+ \in L^1(T^-)$. Then
\begin{itemize}
\item[-] either $f$ has finite dynamical energy
\item[-] or $\mu_f=T^+ \wedge T^-$ charges the pluripolar set $\set{G^++G^-=-\infty}$.
\end{itemize}
\end{theo}

The family of birational maps given in Example \ref{exa:3lines} satisfies this dichotomy. Furthermore,
generic members of this family indeed have finite dynamical energy.

\medskip
\begin{center}
 $\diamond$
\end{center}
\medskip

The structure of the paper is as follows.
In \S \ref{sec:pluripot}, we develop the general framework for intersection of positive closed $(1,1)$ 
currents satisfying the finite energy condition.  Theorem \ref{theo:intersect} and 
a general dichotomy leading to Theorem \ref{theo:eitheror} are both established here.  
In \S \ref{sec:measure}, we consider the nature of the 
measure $\mu_f$ under the assumption that $f$ has finite dynamical energy. Theorem 
\ref{theo:properties} is the end result. 
 In \S \ref{sec:homogeneous}, we consider the particular case of 
homogeneous weights more carefully, giving several criteria for finite dynamical energy along 
the lines of Theorem \ref{theo:criterion}.  Finally,
in \S \ref{sec:examples}, we illustrate our main results by applying them to several significant 
examples, both invertible and non-invertible ones among them.

\medskip

\noindent {\it Acknowledgment.} We would like to thank A.Zeriahi for several
useful discussions.

\section{Weighted energy with respect to a current} \label{sec:pluripot}

We develop in this section the pluripotential theoretic tools that we will use to define and understand the wedge product $T^+ \wedge T^-$.  There is no dynamics here, so these results might accordingly be of independent interest.

\subsection{The class $\E(T,\alpha)$}\label{subs:class}
Throughout the section, we take $\alpha,T$ to be positive closed $(1,1)$ currents on $X$ such that $\alpha$ has bounded potentials.  We assume $\int_X \alpha \wedge T=\{\alpha\} \cdot \{ T \}=1$, the middle term denoting intersection of cohomology classes.  We let $PSH(X,\alpha) = \{u\in L^1(X):\alpha+dd^c u \geq 0\}$ denote the set of \emph{$\alpha$-plurisubharmonic} (or just \emph{$\alpha$-psh}) functions.  If $u\in PSH(X,\alpha)$ for $\alpha$ smooth, then $u$ is called \emph{quasiplurisubharmonic} (or just \emph{qpsh}).

If $u,v \in PSH(X,\alpha)$ are {\it bounded}, it follows from plurifine considerations
(see \cite{bt2}) that
\begin{equation} \label{eq:BT}
{\bf 1}_{\{u>v\}}  [\alpha+dd^c u] \wedge T
={\bf 1}_{\{u>v\}}  [\alpha+dd^c \max(u,v)] \wedge T
\end{equation}
in the sense of Borel measures.

Let $\f \in PSH(X,\alpha)$ be an unbounded function and
$\f_j:=\max(\f,-j) \in PSH(X,\alpha) \cap L^{\infty}(X)$ its
{\it canonical approximants}. We have $\f_j\searrow \f$, and
by (\ref{eq:BT}),
$$
{\bf 1}_{\{ \f>-k \}}  [\alpha+dd^c \f_j] \wedge T
={\bf 1}_{\{ \f>-k \}}  [\alpha+dd^c \f_k] \wedge T
$$
whenever $j \geq k$, since $\set{\f_j>-k}=\set{\f>-k}$ and $\max(\f_j,-k)=\f_k$.
Observe that $\set{\f>-k} \subset \set{\f>-j}$. Hence
$$
\mu_j(\f,T):={\bf 1}_{\{\f>-j\}} [\alpha+dd^c \f_j] \wedge T
$$
is an increasing sequence of Borel measures on $X$, whose total mass
is bounded from above by $1=\{ \alpha \} \cdot \{T \}$.

\begin{defi}
We set $\mu(\f,T):=\lim\nearrow \mu_j(\f,T)$ and
$$
\E (T,\alpha):=\{ \f \in PSH(X,\alpha) \, / \, \mu(\f,T)(X)=1 \}.
$$
\end{defi}

\noindent Alternatively $\f \in \E(T,\alpha)$ if and only if
$[\alpha+dd^c \f_j] \wedge T(\set{\f \leq -j}) \rightarrow 0$.
\vskip.1cm

Observe that the probability measures $[\alpha+dd^c \f_j] \wedge T$ converge to $\mu(\f,T)$ as Borel measures (i.e. in mass) 
when $\f \in \E(T,\alpha)$.  This is much stronger than convergence as Radon measures (i.e. in the weak topology) and 
furnishes the key to the next proposition.

\begin{pro} \label{pro:plp}
Assume $\f \in \E(T,\alpha)$. Then
\begin{enumerate}
\item the measures $\alpha \wedge T$ and $\mu(\f,T)$ do not charge the set $\set{\f=-\infty}$;
\item if $T$ puts no mass on a complete pluripolar set $P$, then neither does  $\mu(\f,T)$.
\end{enumerate}
\end{pro}

\begin{proof}
We can suppose that $\f \leq 0$.
To simplify notation we set $\mu =\mu(\f,T)$ and $\mu_j =\mu_j(\f,T)$.
Note first that $\mu_j(\set{\f=-\infty})=0$, and therefore $\mu(\set{\f=-\infty})=0$.

Fix $\chi:\R \rightarrow \R$ a convex increasing function such that
$\chi(-\infty)=-\infty$ and $\chi \circ \f \in L^1(\mu)$ 
(see Lemma \ref{lem:basicweight} below).
\ignore{By definition $\mu(\f>-j)=\mu_j(\f>-j)$ hence 
$\mu(\f \leq -j)=\mu_j(\f \leq -j)$ since $\mu(X)=\mu_j(X)=1$.}
Since $0\geq \f_j \geq \f$ and $0\leq \mu_j\leq \mu$, we have
$$
\sup_j \int_X (-\chi) \circ \f_j \, d\mu_j \leq \int_X (-\chi) \circ \f \, d\mu<+\infty.
$$
By Stokes theorem,
$$
\int_X (-\chi) \circ \f_j \, (\a +dd^c \f_j) \wedge T = \int  (-\chi) \circ \f_j \, \alpha \wedge T
+ \int \chi' \circ \f_j d\f_j \wedge d^c \f_j \wedge T,
$$
where the rightmost term is non-negative.  
Observe that the measures $(\a+dd^c \f_j) \wedge T$ and $\mu$ have the same mass
since $\f \in {\mathcal E}(T,\a)$ and coincide in $(\f>-j)$, hence
$$
(\a+dd^c \f_j) \wedge T(\f \leq -j) =\mu(\f \leq -j).
$$
Therefore
$$
\int_{(\f \leq -j)} (-\chi) \circ \f_j \, (\a+dd^c \f_j) \wedge T
=(-\chi)(-j) \mu(\f \leq -j) 
\leq 
\int_{(\f \leq -j)} (-\chi) \circ \f \, d\mu,
$$
which yields
$$
\int_X (-\chi) \circ \f_j \, \a \wedge T \leq \int_X (-\chi) \circ \f \, d\mu<+\infty.
$$
We infer $\chi \circ \f \in L^1(\alpha \wedge T)$, hence
in particular $\alpha \wedge T(\f=-\infty)=0$.

Let $P$ be a complete pluripolar set, i.e. $P =\{ \p=-\infty \}$ for some
quasi-plurisubharmonic function $\p \leq 0$ on $X$.
Let $\om$ be a K\"ahler form. By assumption $\om \wedge T(P)=0$.
Attenuating the singularities of $\p$ if necessary (i.e. replacing $\p$
by $\chi \circ \p$ for some convex increasing function $\chi$ with slow growth
at $-\infty$ and such that still $\chi(-\infty)=-\infty$), we can assume
$\p \in L^1(\om \wedge T)$.

Write $\a=\theta+dd^c u$, where $\theta$ is smooth and $u$ is bounded. We can also
assume without loss of generality that $\om'=\theta+\om$ is K\"ahler
and $dd^c \p \geq -\om'$ (replace $\om$ by $A \om$, $A>>1$, if necessary). 

We claim that for every bounded $\a$-psh function $v$, $\p \in L^1([\a+dd^c v] \wedge T)$.
Indeed 
$$
0 \leq \int (-\p) [\a+dd^c v] \wedge T \leq
\int (-\p) \om' \wedge T+\int (u+v) (-dd^c \p) \wedge T.
$$
Now we can assume that $u+v \geq 0$ since these functions are bounded.
The conclusion follows by observing that $-dd^c \p \leq \om'$.

This shows in particular that $\mu_j=[\a+dd^c \f_j] \wedge T(P)=0$, hence
$\mu(P)=0$ since $(\mu_j)$ converges to $\mu$ in the strong sense of Borel measures.
\end{proof}

We now introduce a set of weights.

\medskip
\noindent {\bf Notation.}
{\it We let ${\mathcal W}$ be the set of all convex increasing functions
$\chi:\R \rightarrow \R$ such that $\chi(-\infty)=-\infty$
and $\chi(0)=0$.}
\medskip

A straightforward computation (see \cite{gz2}) shows that if 
$\varphi \in PSH(X,\alpha)$ and $\chi'\circ\varphi\leq 1$, then 
$\chi \circ \varphi  \in PSH(X,\alpha)$.

\medskip

\begin{lem}\label{lem:basicweight}
 Let $\mu$ be a positive measure and $u$ a measurable function which is 
bounded from above and such that $u>-\infty$ $\mu$-a.e. Then there exists 
$\chi\in  \mathcal{W}$ such that $\chi\circ u \in L^1(\mu)$.
\end{lem}

\begin{proof} Assume for simplicity that $u$ is negative.  From the identity 
$$
\int( \chi\circ u) d\mu = \int_0^\infty \mu(\chi\circ u<-t)dt,
$$ 
it is straightforward to construct a piecewise affine $\chi$ that suffices.
\end{proof}

\begin{pro}\label{pro:plp2}
 Assume that $T$ has the form $T= \beta+dd^c \p$, with $\beta$ a positive current 
with bounded potentials,
and $T$ puts no mass on the set $\set{\p=-\infty}$. Then $T$ does not charge pluripolar sets.

Consequently if  $\varphi\in \mathcal{E}(T,\alpha)$, then $\mu(\f,T)$ does not charge
pluripolar sets.
\end{pro}

\begin{proof} 
Fix a K\"ahler form $\omega$. Replacing $\beta$ with  $\beta+\omega$, we may assume 
$\{ \beta \}^2>0$.
If $P\subset X$ is locally pluripolar then by Theorem 7.2 in \cite{gz1}, 
$P \subset \{ u =-\infty \}$ for some $u \in PSH(X,\beta)$.  
Fix $\chi \in {\mathcal W}$ such that $\chi \circ \p \in L^1(T \wedge \omega)$. 

We can assume without loss of generality
that $\chi \circ u \in PSH(X,\beta)$ and 
$\int_X (-\chi) \circ u \, [\beta+dd^c u] \wedge \omega <+\infty$, replacing
$u$ by $\chi \circ (u-C)$, $C>0$, if necessary (in other words it is no loss of generality
to assume $u$ has ``small singularities'', see \cite{cgz} for 
more detail).  
The comparison
principle (see Proposition 2.5 in \cite{gz2}) then gives
$$
\int (-\chi) \circ u \, T \wedge \omega \leq
2 \int (-\chi) \circ u \, [\beta+dd^c u]  \wedge \omega
+2 \int (-\chi) \circ \p \, T \wedge \omega
<+\infty.
$$
Therefore $\chi \circ u \in L^1(T \wedge \omega)$, in particular
$T \wedge \omega(P) \leq T \wedge \omega (\set{u=-\infty})=0$.

The second statement of the proposition follows from Proposition \ref{pro:plp}.
\end{proof}

We introduce another class of $\alpha$-psh functions.

\begin{defi}
For $\chi \in {\mathcal W}$ we set
$$
\E_{\chi}(T,\alpha):=\left\{ \f \in PSH(X,\alpha) \, / \, 
\sup_{j \in \N} \int_X  (-\chi) \circ \f_j [\alpha+dd^c \f_j] \wedge T <+\infty \right\},
$$
where $\f_j:=\max(\f,-j)$ are the canonical approximants.
\end{defi}

The relationship between these classes and $\E(T,\alpha)$ is the following.

\begin{pro}\label{prop:union} We have 
$$
\E(T,\alpha)=\bigcup_{\chi \in {\mathcal W}} \E_{\chi}(T,\alpha).
$$
\end{pro}

\begin{proof}
Given $\f \in \E_{\chi}(T,\alpha)$, we may assume $\f \leq 0$.  Then
$$ 
0 \leq [\alpha+dd^c \f_j] \wedge T (\set{\f \leq -j}) \leq
\frac{1}{|\chi(-j)|} \sup_k \int_X (-\chi \circ \f_k) [\alpha+dd^c \f_k] \wedge T
\rightarrow 0,
$$
since $\chi(-\infty)=-\infty$. Hence $\f \in \E(T,\alpha)$.

Conversely assume $\f \in \E(T,\alpha)$. By Proposition \ref{pro:plp} and Lemma \ref{lem:basicweight}, there is 
weight $\chi \in {\mathcal W}$ such that $\chi \circ \f \in L^1(\mu(\f,T))$. We have by definition
\begin{equation} \label{eq:defmu}
\mu(\f,T)(\set{\f>-j})=[\alpha+dd^c \f_j] \wedge T(\set{\f>-j}).
\end{equation}

Since the measures on either side have the same total mass, we infer that 
\begin{equation} \label{eq:mumass}
\mu(\f,T)(\set{\f \leq -j})=[\alpha+dd^c \f_j] \wedge T(\set{\f \leq -j}).
\end{equation}

Writing $\mu(\f)=\mu(\f,T)$ and $\alpha_{\f_j}:=\alpha+dd^c \f_j$, we use \eqref{eq:defmu} again to get
\begin{eqnarray*}
\int_X (-\chi \circ \f_j) \, \alpha_{\f_j} \wedge T
&=&  (-\chi)(-j) \int_{\set{\f \leq -j}} \, \alpha_{\f_j} \wedge T
+\int_{\set{\f>-j}} (-\chi \circ \f) \, \alpha_{\f_j} \wedge T \\
&=& (-\chi)(-j) \int_{\set{\f \leq -j}} \, d \mu(\f)
+\int_{\set{\f>-j}} (-\chi \circ \f)  \, d \mu(\f) \\
& \leq & \int_X (-\chi \circ \f) \, d \mu(\f).
\end{eqnarray*}
So $\f \in \E_{\chi}(T,\alpha)$.
\end{proof}

\subsection{Intersection of currents}
Let $\alpha,\beta,S,T$ be positive closed $(1,1)$ currents 
on $X$ such that
\begin{enumerate}
\item $\alpha,\beta$ have bounded potentials;
\item $S=\alpha+dd^c \f$ for some $\f \in PSH(X,\alpha)$;
\item $T=\beta+dd^c \p$ for some $\p\in PSH(X,\beta)$;
\item $\{\alpha\} \cdot \{\beta\}=\int_X \alpha \wedge \beta=\{S\} \cdot \{T\}=1$.
\end{enumerate}

We want to define the wedge product $S \wedge T$.  That is we want
to construct a probability measure $\mu$ such that whenever 
$\f_j \in PSH(X,\alpha) \cap L^{\infty}(X)$, 
$\p_j \in PSH(X,\beta) \cap L^{\infty}(X)$ decrease towards $\f,\p$ then
\begin{equation} \tag{\dag}
S_j \wedge T_j:=[\alpha+dd^c \f_j] \wedge [\beta \wedge dd^c \p_j] \longrightarrow \mu
\end{equation}
in the weak sense of Radon measures. Recall that the wedge product $S_j \wedge T_j$ of closed 
currents with bounded potentials is well defined thanks to the work of 
E.Bedford and A.Taylor \cite{bt1}.  
It is well known that it is not always possible to define $S\wedge T$ when the currents have unbounded 
potentials, even if $S=T$.  We show here that it is nevertheless possible in some very general situations.

\subsubsection*{The $L^{1}$ condition}

When the potential $\f$ of $S$, is integrable with respect to (the trace measure of) the current $T$, then the current 
$\f T$ is well-defined, as is therefore 
$$
\mu=S \wedge T:=\alpha \wedge T +dd^c (\f T).
$$
It is well known that the continuity property $(\dag)$ holds in this case.
This is a consequence of the following lemma, which we will need for own purposes.

\begin{lem} \label{lem:tilde}
Assume $\f \in L^1(T)$. Then for any bounded quasiplurisubharmonic function $u$,
and for any sequence $\f_j \in PSH(X,\alpha)$ decreasing to $\f$, one has
$\f_j \in L^1(T)$ and
$$
\int_X u \, [\alpha+dd^c \f_j] \wedge T \longrightarrow \int_X u \, [\alpha+dd^c \f] \wedge T.
$$
\end{lem}

\begin{proof}
By the monotone convergence theorem $\f_jT \to \f T$ as currents.
Therefore $[\alpha+dd^c \f_j] \wedge T \rightarrow [\alpha+dd^c \f] \wedge T$ in the weak sense of 
Radon measures. Hence we are done if $u$ is {\it continuous}.
Since $u$ is upper semi-continuous, we get that 
$$
\limsup \int_X u \, [\alpha+dd^c \f_j] \wedge T \leq  \int_X u \, [\alpha+dd^c \f] \wedge T.
$$

We now use the assumption that $u \in PSH(X,\gamma)$ for some positive closed
$(1,1)$-current $\gamma$ with bounded potentials.
It follows from repeated application of Stokes theorem that
\begin{eqnarray*}
\int_X u \, [\alpha+dd^c \f_j] \wedge T &=&
\int u \, \alpha \wedge T +\int_X \f_j \, [\gamma+dd^c u] \wedge T -\int_X \f_j \gamma \wedge T \\
& \geq & \int_X u \, [\alpha+dd^c \f] \wedge T+\int_X (\f-\f_j) \gamma \wedge T.
\end{eqnarray*}
The integrations by parts are easily justified because $u$ is bounded.
By monotone convergence, $\int_X (\f-\f_j) \gamma \wedge T \rightarrow0$.  So we infer 
$$
\liminf \int_X u \, [\alpha+dd^c \f_j] \wedge T \geq  \int_X u \, [\alpha+dd^c \f] \wedge T.
$$
\end{proof}

Note that the condition $\f \in L^1(T)$ is {\it symmetric},
$\f \in L^1(T) \Leftrightarrow \p \in L^1(S)$. This follows from the Stokes theorem:
if $\omega$ is any fixed K\"ahler form, then
$$
\int_X \f T \wedge \omega=\int_X \p S \wedge \omega+
\int_X \f \beta \wedge \omega-\int_X \p \alpha \wedge \omega,
$$
where the last two integrals are  finite because qpsh functions are always integrable with 
respect to  measures of the form $\alpha\wedge \omega$ (resp. $\beta \wedge \omega$)
(see e.g. \cite{demailly}).
 
Lastly, note that if $\alpha_{1}+ dd^c\f_1$ and 
$\alpha_2+dd^c\f_2$ are two representations of the same closed positive current $S$, 
with $\alpha_{i}$ of bounded potential, then $\f_1\in L^1(T)$ iff $\f_2\in L^1(T)$. 
It therefore makes sense to write ``$S\in L^1(T)$" as a shorthand for ``$\f\in L^1(T)$ for 
some choice of $\alpha$".  Thus we have just seen that 
$$
S\in L^1(T) \Leftrightarrow T\in L^1(S).
$$ 

In the next paragraph we will give a different approach to the wedge product, 
using the energy formalism.

\subsubsection*{The energy condition}
We will show that the wedge product $S \wedge T$ can be defined 
as $\mu(\f,T)$ whenever $\f \in \E(T,\alpha)$.

\begin{pro} \label{pro:alternative}
Assume $\f \in L^1(T)$ so that the probability measure $S \wedge T$ is well defined.
Then
$$
\mu(\f,T)={\bf 1}_{\{ \f > -\infty \}} S \wedge T.
$$
Therefore $\f \in \E(T,\alpha)$ if and only if $S \wedge T(\set{\f=-\infty})=0$.
\end{pro}

\begin{proof}Without loss of generality, assume that $\f<0$.
Let $\f_j:=\max(\f,-j)$ and $u_k:=\max(\f/k+1,0)\in PSH(X,\alpha) \cap L^{\infty}(X)$, 
where $k \leq j$ is fixed. Observe that $\set{u_k>0}=\set{\f>-k}$ and $u_k=0$ elsewhere,
thus
$$
u_k \, [\alpha+dd^c \f_j] \wedge T=
u_k {\bf 1}_{\{\f>-j\}}\, [\alpha+dd^c \f_j] \wedge T=
u_k \, \mu_j(\f,T),
$$

Letting $j \rightarrow +\infty$ we infer, by using Lemma \ref{lem:tilde}, that
$$
\langle u_k \, S \wedge T, h \rangle = \langle u_k \mu(\f,T), h \rangle,
\text{ for all } k \in \N
$$
and for any continuous test function $h$ on $X$.
Now $u_k \nearrow {\bf 1}_{\{\f>-\infty\}}$, so
$$
\langle {\bf 1}_{\{\f>-\infty \}} \, S \wedge T, h \rangle = 
\langle \mu(\f,T), h \rangle
$$
since $\mu(\f,T)$ does not charge $\set{\f=-\infty}$.
\end{proof}

Proposition \ref{pro:alternative} implies that whenever 
$S\in L^1(T)$, one has $\f \in \E(T,\alpha)$ if and only if 
$\psi\in \E(S,\beta)$, and in either case $\mu(\f,T) = \mu(\psi,S)$.
The next result gives some symmetry even without the integrability assumption.  
It follows from a monotone convergence argument that we defer until the next subsection.

\begin{thm} \label{thm:cvma}
Assume $\f \in \E(T,\alpha)$ and $\p \in \E(S,\beta)$. 
Assume moreover that $T$ does not charge pluripolar sets. Then 
$$
\mu(\f,T)=\mu(\p,S)
$$
and this measure does not charge pluripolar sets.

If, moreover, $\f_j \in PSH(X,\alpha)$, $\p_j \in PSH(X,\beta)$
decrease to $\f$, $\p$, and we set $S_j:=\alpha+dd^c \f_j$, $T_j:=\beta+dd^c \p_j$,
then we have $\f_j \in \E(T_j,\alpha), \p_j \in \E(S_j,\beta)$ and
$$
\mu(\f_j,T_j)=\mu(\p_j,S_j) \longrightarrow \mu(\f,T)=\mu(\p,S)
$$
in the weak sense of Radon measures.
\end{thm}

Note that the condition $\f\in \mathcal{E}(T,\alpha)$ does not depend on the choice of $\alpha$.

\begin{pro}
Assume that $S=\alpha_{1}+ dd^c\f_1=\alpha_2+dd^c\f_2$, where $\alpha_1$, $\alpha_2$ 
have bounded potentials. Then $\mu(\f_1, T) = \mu(\f_2,T)$.
In particular $\f_1\in \mathcal{E}(T,\alpha_1)$ if and only if $\f_2\in \mathcal{E}(T,\alpha_2)$.
\end{pro}

\begin{proof}
Fix $u$ bounded such that $\a_2=\a_1+dd^c u$. Subtracting a constant if necessary,
we can assume $\f_1=\f_2+u$. Fix $M>0$ such that $-M \leq u \leq +M$.

Observe that $\max(\f_2,-j)+u=\max(\f_1,-j+M)$ in the plurifine open set
$\set{\f_1>-j+M} \subset \set{\f_2>-j}$. Thus
\begin{eqnarray*}
\lefteqn{
\1_{ \{\f_1>-j+M\} } [\a_1+dd^c \max(\f_1,-j+M) ] \wedge T}
\hskip2cm \\
&&\leq  \1_{ \{\f_2>-j\} } [\a_2+dd^c \max(\f_2,-j) ] \wedge T
\leq  \mu(\f_2,T).
\end{eqnarray*}

We infer $\mu(\f_1,T) \leq \mu(\f_2,T)$, whence equality by reversing the roles of
$\f_1$ and $\f_2$. In particular $\f_1\in \mathcal{E}(T,\alpha_1)$ if and only if 
$\f_2\in \mathcal{E}(T,\alpha_2)$.
\end{proof}

We already know that if both $\f_j$ and $\p_j$ are bounded, then $\mu(\f_j,T_j)=S_j \wedge T_j$.  
Hence under the hypotheses of Theorem \ref{thm:cvma}, our results now make it reasonable to set
$$
S \wedge T:=\mu(\f,T)=\mu(\p,S).
$$
Also we can write ``$S\in \mathcal{E}(T)$'' to say that $\f\in \mathcal{E}(T,\alpha)$ 
for some choice of $\alpha$. 
In summary, we have shown that if $S\in  \mathcal{E}(T)$ and $T\in \mathcal{E}(S)$, 
then there is a well-defined 
wedge product $S\wedge T$.  We stress that our definition of $S \wedge T$ applies even in some cases
where $S \notin L^1(T)$ (see \cite[\S 2.4]{gz2}). On the other hand, the case where $S$ and $T$ are transversely 
intersecting lines shows that $S\in L^1(T)$ does not mean that $S\in \mathcal{E}(T)$.

\medskip

We now discuss the hypothesis  made in
Theorem \ref{thm:cvma} that $T$ does not charge pluripolar sets. Later on we will apply  this construction to the case
where $T=T^+$ is the canonical $f^*$-invariant current associated to
a $1$-stable endomorphism $f:X \rightarrow X$ of small topological degree.
The laminar structure and extremality properties of $T^+$ will allow us
to reach the following alternative:
\begin{itemize}
\item[-] either $T^+$ does not charge pluripolar sets, 
\item[-] or $T^+$ is supported on a complete pluripolar set.
\end{itemize}

It is therefore important to notice that the latter possibility
can not occur under the finite energy assumption.

\begin{pro} \label{pro:degen}
Assume $T=\beta+dd^c \p$ is supported on $(\p=-\infty)$
and $\f \in {\mathcal E}(T,\alpha)$. Then $\mu(\p,S)=0$, hence
in particular $\p \notin {\mathcal E}(S,\beta)$.
\end{pro}

\begin{proof}
Assume $\f \in {\mathcal E}(T,\alpha)$. It follows from Lemma \ref{lem:ineqfdtal2}
that $\f \in {\mathcal E}(T_j,\a)$, where $T_j=\beta+dd^c \p_j$,
$\p_j=\max(\p,-j)$.  

Observe that by definition of $\mu(\f,T_j)$, the measures $T_j \wedge S_k$
converge (in the strong sense of Borel measures) towards $\mu(\f,T_j)$,
as $k \rightarrow +\infty$. Here $S_k=\a+dd^c \f_k$, where $\f_k=\max(\f,-k)$.

Now $\mu(\f,T_j)=T_j \wedge S$, as follows from Proposition \ref{pro:alternative}.
We infer
$$
{\bf 1}_{\{\p>-j\}} T_j \wedge S_k \stackrel{k \rightarrow +\infty}{\longrightarrow}
{\bf 1}_{\{\p>-j\}} T_j \wedge S.
$$
Observe finally that 
${\bf 1}_{\{\p>-j\}} T_j \wedge S_k={\bf 1}_{\{\p>-j\}} T \wedge S_k=0$
if $T$ is supported on $(\p=-\infty)$. 
The latter equality follows from lemma \ref{lem:folklore} below.
Thus
${\bf 1}_{\{\p>-j\}} T_j \wedge S=0$, hence $\mu(\p,S)=0$.
\end{proof}

The following result is probably known to experts in pluripotential theory.
Since we could not find a reference,
we include a proof.

\begin{lem} \label{lem:folklore}
Assume $T=dd^c \p \geq 0$ is a positive closed $(1,1)$ current
in the unit ball ${\bf B} \subset \C^2$, which gives full mass to  $\set{\p=-\infty}$.
Then so does the measure $T \wedge dd^c u$, for any
locally bounded plurisubharmonic function $u$.
\end{lem}

\begin{proof}
This is a local question; we can assume all our objects are defined
in a small neighborhood of $\overline{{\bf B}}$.
Set $\r:=e^{\p}$. This is a bounded psh function such that
$\set{\r >0}=\set{\p>-\infty}$ and $\set{\r=0}=\set{\p=-\infty}$. By assumption
$\r T =0$ and we need to prove that $\r \mu=0$, where
$\mu=T \wedge dd^c u$, $u \in PSH({\bf B}) \cap L_{loc}^{\infty}({\bf B})$.

If $u$ is smooth, this easily follows from the identity
$$
\langle \r \mu, \chi \rangle =\langle \r T, \chi dd^c u \rangle
$$
valid for any test function $\chi$.

For the general case we approximate $u$ by a decreasing sequence
of smooth psh functions $u_j$. Set $\mu_j:=T \wedge dd^c u_j$.
Since 
$$
\langle \r \mu_j, \chi \rangle =\langle \r T, \chi dd^c u_j \rangle=0,
$$
it suffices to show that the measures $\r \mu_j$ converge, in the weak sense
of Radon measures, towards $\r \mu$.
This is obvious if $\r$ is continuous. 
Let $\r_{\e}$ denote a sequence of smooth psh functions decreasing to $\r$.
Since 
$$
\lim_{j \rightarrow +\infty}\langle \r_{\e} \mu_j, \chi \rangle 
=
\langle \r_{\e} \mu, \chi \rangle,
$$
for fixed $\e>0$, it suffices to show that 
$\r_{\e} \mu_j \rightarrow \r \mu_j$ as $\e \rightarrow 0$ 
uniformly with respect to $j$.

Using a ``max-construction'', we can assume without loss of generality that
$\r_{\e} \equiv \r \equiv u_j \equiv u \equiv ||z||^2 -1$ near
$\partial {\bf B}$. This allows us to integrate by parts, since all these functions
vanish on $\partial {\bf B}$.
Let $\chi$ be a test function, then
\begin{eqnarray*}
\left| \langle \r_{\e} \mu_j, \chi \rangle -\langle \r \mu_j, \chi \rangle \right|
&\leq& 
||\chi||_{{\mathcal C}^2} \langle (\r_{\e}-\r) dd^c u_j, T \rangle \\
&=& C_{\chi} \int d(\r_{\e}-\r) \wedge d^c u_j \wedge T \\
&\leq& C_{\chi}
\left( \int d(\r_{\e}-\r) \wedge d^c (\r_{\e}-\r) \wedge T \right)^{1/2}
\left( \int d u_j \wedge d^c u_j \wedge T \right)^{1/2},
\end{eqnarray*}
as follows from Cauchy-Schwarz inequality.

The latter integral is uniformly bounded from above,
$$
\int d u_j \wedge d^c u_j \wedge T=
\int (-u_j) \wedge dd^c u_j \wedge T 
\leq ||u||_{L^{\infty}} \int dd^c u_j \wedge T \leq M_0,
$$
while the next to last converges to zero,
$$
\int d(\r_{\e}-\r) \wedge d^c (\r_{\e}-\r) \wedge T
\leq \int (\r_{\e}-\r) dd^c \r \wedge T \longrightarrow 0,
$$
as follows from the monotone convergence theorem.
\end{proof}

\medskip

\subsection{Proof of Theorem \ref{thm:cvma}}

We start with two useful inequalities.

\begin{lem} \label{lem:ineqfdtal}
Fix $\chi \in {\mathcal W}$ and let $u,v \in PSH(X,\alpha) \cap L^{\infty}(X)$ be such
that $u \leq v \leq 0$. Then
$$
0 \leq \int_X (-\chi \circ v) \, [\alpha+dd^c v] \wedge T
\leq 2 \int_X (-\chi \circ u) \, [\alpha+dd^c u] \wedge T.
$$
\end{lem}

The proof is a simple integration by parts (see Lemma 2.3 in \cite{gz2} for similar
computation). It will follow from this lemma
 that in the definition of the class $\E_{\chi}(X,\alpha)$,
one can replace the canonical approximants by {\it any} sequence of bounded
$\alpha$-psh functions decreasing towards $\f$.

\medskip

Fix $\gamma \geq 0$ a positive closed current with bounded potentials.
For this lemma we use the notation $\alpha_{\f}:=\alpha+dd^c \f$ and $\gamma_u:=\gamma+dd^c u$.

\begin{lem} \label{lem:ineqfdtal2}
Fix $\chi \in {\mathcal W}$ and $0 \geq \f \in PSH(X,\alpha) \cap L^{\infty}(X)$.
Let $u \leq v$ be two  $\gamma$-psh functions. Then
$$
0 \leq \int_X (-\chi \circ \f) \, \alpha_{\f} \wedge \gamma_v
\leq 2 \int_X (-\chi \circ \f) \, \alpha_{\f} \wedge \gamma_u+ \chi'(0) \int_X [v-u] \alpha^2.
$$
\end{lem}

\begin{proof}
It follows from Stokes theorem that
$$
\int (-\chi \circ \f) \, \alpha_{\f} \wedge \gamma_v=
\int (-\chi \circ \f) \, \alpha_{\f} \wedge \gamma_u+ \int (v-u) \alpha_{\f} \wedge [-dd^c \chi \circ \f].
$$
Observe that $-dd^c \chi \circ \f \leq \chi' \circ \f \, \alpha$, thus the latter integral
is bounded from above by $I=\int (v-u) \chi' \circ \f \alpha_{\f} \wedge \alpha$. Now
$$
\chi' \circ \f \, \alpha_{\f} \leq \chi' \circ \f \, \alpha_{\f}+\chi'' \circ \f \, d\f \wedge d^c \f
=\chi' \circ \f \, \alpha+dd^c (\chi \circ \f),
$$
so we may estimate the integral $I$
\begin{eqnarray*}
I &\leq& \int (-\chi \circ \f) \, \alpha \wedge dd^c(u-v) +\int (v-u) \chi' \circ \f \, \alpha^2 \\
&\leq& \int (-\chi \circ \f) \, \alpha \wedge \gamma_u +\chi'(0) \int (v-u) \alpha^2.
\end{eqnarray*}
The conclusion follows by observing that
$$
\int (-\chi \circ \f) \, \alpha_\f \wedge \gamma_u=
\int (-\chi \circ \f) \, \alpha \wedge \gamma_u+
\int \chi' \circ \f \, d\f \wedge d^c \f \wedge \gamma_u
\geq \int (-\chi \circ \f) \, \alpha \wedge \gamma_u.
$$
\end{proof}

\begin{proof}[Proof of Theorem \ref{thm:cvma}]
Without loss of generality we can assume that  $\f,\p \leq 0$.

\smallskip

\noindent{\em Step 1}.
Assume that $\f_j:=\max(\f,-j)$ and $\p_j:=\max(\p,-j)$ are the canonical approximants.
We first show that the measures $S_j \wedge T_j$ converge to $\mu(\f,T)$.
Recall that by definition,
$$
\mu(\f,T)=\lim \nearrow {\bf 1}_{\{ \f>-j \} } S_j \wedge T.
$$
Fix $N \in \N$. It follows from (\ref{eq:BT}) that for all $j \geq N$,
$$
S_j \wedge T_j \geq {\bf 1}_{\{ \p>-N \} } S_j \wedge T_j
={\bf 1}_{\{ \p>-N \} } S_j \wedge T.
$$
Let $\sigma$ be a cluster point of the sequence $(S_j \wedge T_j)$.
We infer
$$
\sigma \geq \lim_{N \rightarrow +\infty} {\bf 1}_{\{ \p >-N\}} \mu(\f,T)
={\bf 1}_{\{ \p >-\infty \}} \mu(\f,T)=\mu(\f,T),
$$
since $\mu(\f,T)$ does not charge the pluripolar set 
$\set{\p=-\infty}$, as follows from Proposition 
\ref{pro:plp} (because we assume $T$ does not charge pluripolar sets).
Since both $\sigma$ and $\mu(\f,T)$ are probability measures,
it follows that $\sigma=\mu(\f,T)$.

\medskip

We now show that $\mu(\f,T)=\mu(\p,S)$.
Observe first that
$$
{\bf 1}_{\{\f>-j \} \cap \{ \p>-j \} } S_j \wedge T_j
\longrightarrow \mu(\f,T),
$$
where the convergence holds in the strong sense of Borel measures.
Recall that $\mu(\p,S)=\lim \nearrow {\bf 1}_{\{ \p>-j \} } S \wedge T_j$.
Now ${\bf 1}_{\{\f>-j \} } S \wedge T_j={\bf 1}_{\{\f>-j \} } S_j \wedge T_j$, hence
$$
{\bf 1}_{\{ \f>-j \} }
\left[ {\bf 1}_{\{ \p>-j \} } S \wedge T_j \right]=
{\bf 1}_{\{\f>-j \} \cap \{ \p>-j \} } S_j \wedge T_j.
$$

We infer ${\bf 1}_{\{ \f>-\infty \}} \mu(\p,S)=\mu(\f,T)$.
Since these are both probability measures, we conclude that  $\mu(\p,S)=\mu(\f,T)$ 
and this measure does not charge pluripolar sets (by Proposition \ref{pro:plp}).

\medskip

\noindent{\em Step 2.} In the sequel we set $\mu=\mu(\f,T)=\mu(\p,S)$ and
we fix $\chi \in {\mathcal W}$ such that $\f \in \E_{\chi}(T,\alpha)$ and
$\p \in \E_{\chi}(S,\beta)$.
We  show  that
$$
\chi \circ \f_j \, S_j \wedge T_j \longrightarrow \chi \circ \f \mu
$$
in the strong sense of Borel measures.

Let $B$ be a Borel subset of $X$. We leave the reader check that
$\int_B \chi \circ \f_j S_j \wedge T$ converges to $\int_B \chi \circ \f d\mu$ as $j\cv \infty$.
It then suffices to verify that
$\int_B \chi \circ \f_j \, S_j \wedge (T-T_j) \rightarrow 0$.
It follows  from (\ref{eq:BT}) that
$$
\1_{\{\p>-j \}} \chi \circ \f_j S_j \wedge T_j 
=\1_{\{\p>-j \}} \chi \circ \f_j S_j \wedge T.
$$
Since
$$
\int_{B \cap \{\p \leq -j\}} |\chi| \circ \f_j S_j \wedge T
\leq \int_{\{\p \leq -j\}} |\chi| \circ \f_j S_j \wedge T
\rightarrow \int_{\{\p =-\infty\}} |\chi| \circ \f d \mu=0,
$$
we will be done if we can show that
$\int_{\{\p \leq -j\}} |\chi| \circ \f_j S_j \wedge T_j \rightarrow 0$.

>From Lemmas \ref{lem:ineqfdtal} and \ref{lem:ineqfdtal2}, we obtain $M_\chi\in\R$ such that 
$$
\int (-\chi) \circ \f_j \, S_j \wedge T_j \leq M_{\chi}, \qquad \int (-\chi) \circ \p_j \, S_j \wedge T_j \leq M_\chi\quad\text{ for all } j\in\N.
$$
Choose another weight $\tilde{\chi} \in {\mathcal W}$, such that the same uniform bound hold, 
and such that   moreover   $\chi=o(\tilde{\chi})$ and $\chi/\tilde\chi$ is increasing 
(to find such a $\tilde \chi$, choose $\tilde \chi$ first and then $\chi$!). 
 We conclude that 
$$
\int_{\set{\p \leq -j}} |\chi| \circ \f_j S_j \wedge T_j \leq
M_{\tilde{\chi}} \left| \frac{\chi}{\tilde{\chi}}(-j) \right| \rightarrow 0,
$$
as desired.

>From {\it Step 2} we immediately obtain the following generalization of Lemma \ref{lem:ineqfdtal}.

\begin{cor}\label{cor:ineqfdtal}
Let $\chi \in {\mathcal W}$, $u\in \mathcal{E}_{\chi}(T,\alpha)$ and $v\in PSH(X,\alpha)$   such
that $u \leq v \leq 0$. Then  $v\in \mathcal{E}_{\chi}(T,\alpha)$ and 
$$
0 \leq \int_X (-\chi \circ v) \, [\alpha+dd^c v] \wedge T
\leq 2 \int_X (-\chi \circ u) \, [\alpha+dd^c u] \wedge T.
$$
\end{cor}

\medskip

\noindent{\em Step 3}. 
Let $\f_j,\p_j \leq 0$ denote now {\it arbitrary} sequences of $\alpha$-psh, $\beta$-psh functions decreasing towards $\f,\p$.  From Lemma \ref{lem:ineqfdtal2} and Corollary \ref{cor:ineqfdtal}, we infer that  $\f_j \in \E_{\chi}(T_j,\alpha)$ and $\p_j \in \E_{\chi}(S_j,\beta)$.  Thus the measures $\mu(\f_j,T_j)=\mu(\p_j,S_j)$ are well defined. We set
$$
\f_j^{(K)}:=\max(\f_j,-K), \;
\p_j^{(K)}:=\max(\p_j,-K), \;
\f^{(K)}:=\max(\f,-K)
$$
and $\p^{(K)}:=\max(\p,-K)$. Similarly
$$
S_j^{(K)}:=\alpha+dd^c \f_j^{(K)}, \; 
T_j^{(K)}:=\beta+dd^c \p_j^{(K)}, \;
S^{(K)}=\alpha+dd^c \f^{(K)}
$$
and $T^{(K)}=\beta+dd^c \p^{(K)}$. 

It follows from {\it Step 1} that $S^{(K)} \wedge T^{(K)} \rightarrow \mu$
and $S_j^{(K)} \wedge T_j^{(K)} \rightarrow S_j \wedge T_j$ when $K \rightarrow +\infty$.
It follows from the monotone convergence theorem of \cite{bt1}
that for each fixed $K$,
$$
S_j^{(K)} \wedge T_j^{(K)} \longrightarrow S^{(K)} \wedge T^{(K)}
\text{ as } j \rightarrow +\infty.
$$
Thus we will be done if we can prove that the convergence
$S_j^{(K)} \wedge T_j^{(K)} \rightarrow S_j \wedge T_j$ is uniform with respect
to $j$. This is what we  show now.

Fix  a Borel subset $B \subset X$. Observe that
$\set{\f_j \leq -K} \cup \set{\p_j \leq -K} \subset \set{u_j \leq -K}$, where $u_j:=\f_j+\p_j$.
Fixing $\chi \in {\mathcal W}$ as at the end of {\it Step 2}, we have from convexity of $\chi$ that
$$
\sup_j \int_X (-\chi \circ u_j) \, S_j \wedge T_j \leq \sup_j \int_X (-\chi \circ \f_j) \, S_j \wedge T_j +
\sup_j\int_X (-\chi \circ \p_j) \, S_j \wedge T_j
\leq 2M_{\chi}.
$$
It follows again from (\ref{eq:BT}) that $S_j^{(K)} \wedge T_j^{(K)} \equiv S_j \wedge T_j$
in the plurifine open set $\set{\f_j>-K} \cap \set{\p_j >-K}$, thus
$$
\left| S_j^{(K)} \wedge T_j^{(K)}(B)-S_j \wedge T_j(B) \right|
\leq \int_{\set{u_j \leq -K}} \left[ S_j^{(K)} \wedge T_j^{(K)}+S_j \wedge T_j \right].
$$
Lemma \ref{lem:ineqfdtal2} and Corollary \ref{cor:ineqfdtal} imply that $\int (-\chi) \circ u_j^{(K)} S_j^{(K)} \wedge T_j^{(K)}$ is bounded above by $4M_{\chi} + C$.  Hence
$$
\left| S_j^{(K)} \wedge T_j^{(K)}(B)-S_j \wedge T_j(B) \right| 
\leq \frac{6 M_{\chi}+C}{|\chi(-K)|}
$$
converges to zero as $K \rightarrow +\infty$, uniformly with respect to $j$.
\end{proof}

\subsection{The gradient approach}

We now give an alternative description of the finite energy conditions in terms
of integrability properties of weighted gradients (in the spirit of \cite{bd} who
considered the special case $\chi(t)=t$ in what follows). 
\vskip.1cm

Recall that any function $\f \in PSH(X,\alpha)$ has gradient in $L^{2-\e}(X)$ for any $\e>0$,
but $\nabla \f \notin L^2(X)$ in general. Indeed, $\f$ has gradient in $L^2$ if and only if $\f\in L^1(\alpha_{\f})$, where we write $\alpha_\f = \alpha+dd^c\f$ as before. More generally when $\alpha_{\f}:=\alpha+dd^c \f$
does not charge the set $\set{\f=-\infty}$, there exists $\chi \in {\mathcal W}$ such that
$\chi \circ \f \in L^1(\alpha_{\f})$.  Hence
$$
\int_X \chi' \circ \f \, d\f \wedge d^c \f \wedge \omega
=\int_X (-\chi \circ \f) \, dd^c \f \wedge \omega
\leq \int_X (-\chi \circ \f) \, \alpha_\f \wedge \omega<+\infty,
$$
where $\omega$ is a fixed K\"ahler form.
We get that  $\f$ has weighted gradient in $L^2(X)$.

This suggests that we can give an alternative description of our energy conditions
in terms of weighted gradients. 

\begin{defi}
For $\chi \in {\mathcal W}$ we set
$$
\nabla_{\chi}(T,\alpha):=\left\{ \f \in PSH(X,\alpha) \, / \, 
\sup_{j \geq 0} \int_X \chi' \circ \f_j d\f_j \wedge d^c \f_j \wedge T <+\infty \right\},
$$
where $\f_j:=\max(\f,-j)$ are the canonical approximants.
\end{defi}

\begin{pro} \label{pro:grad}
Fix $\chi \in {\mathcal W}$. Then
$$
\E_{\chi}(T,\alpha)=\left\{ \f \in \nabla_{\chi}(T,\alpha) \, / \, \chi \circ \f \in L^1(\alpha \wedge T) \right\}
$$
\end{pro}

\begin{proof}
This follows from integrating by parts,
$$
\int (-\chi \circ \f_j) \, \alpha_{\f_j} \wedge T
=\int (-\chi \circ \f_j) \, \alpha \wedge T+
\int \chi' \circ \f_j \,d\f_j \wedge d^c \f_j \wedge T.
$$
\end{proof}

\begin{rqe}
\label{rqe:nrgl1}
If $\chi'$ is bounded above, then to verify the condition $\chi\circ \f\in L^1(\alpha\wedge T)$ in Proposition \ref{pro:grad}, it suffices to show simply that $\chi\circ \f\in L^1(T)$.
\end{rqe}

\begin{proof}
Since $\alpha\geq 0$ has bounded potentials, we have $\alpha \leq c\omega + dd^c u$ where $u$ is bounded and, without loss of generality, positive.  Thus,
$$
0 \leq \int -\chi\circ \varphi\,\alpha\wedge T \leq c\int -\chi\circ \varphi\, \omega \wedge T + \int -u\,dd^c(\chi\circ\varphi)\wedge T.
$$
The assertion therefore follows from the bound
$
-u\,dd^c(\chi\circ\varphi) \leq -u(\chi'\circ\varphi) dd^c\varphi\leq \norm{u\chi'}_\infty\omega^+.
$

\end{proof}

\subsection{Examples}

We give here simple criteria which ensure that some of our energy conditions
are satisfied.  We let $\E^1(T,\alpha)$ denote the class $\E_{\chi}(T,\alpha)$
for the weight $\chi(t)=t$. Observe that
$$
\E^1(T,\alpha)=\bigcap_{\chi \in {\mathcal W}} \E_{\chi}(T,\alpha).
$$

\begin{pro}
If $\f \in PSH(X,\alpha)$ is bounded, then $\f \in \E^1(T,\alpha)$.
\end{pro}

This is easy and  follows directly from the definitions.  
In dynamical situations, invariant currents with bounded potentials appear for instance for 
 meromorphic maps on surfaces of Kodaira dimension zero 
 (see \S \ref{sec:examples}).  The next result is a bit more elaborate and will
 be useful in particular with  $\varphi =\Gamma^{\pm}$ in \S \ref{sec:homogeneous}.

\begin{pro} \label{pro:energycriter}
Assume $\f \in PSH(X,\alpha) \cap L^{\infty}_{loc}(X \setminus F)$,
where $F$ is a finite set of points. Then $\f \in L^1(T)$.

Moreover if $\f \geq A \log \dist(\cdot,F)-A$ for some constant $A>0$,
and if $\nu(T,p)=0$ for all $p \in F$, then
$\f \in \E(T,\alpha)$.
\end{pro}

\begin{proof}
We can assume without loss of generality that $\f \leq 0$ on $X$.
Fix $\omega$ a K\"ahler form on $X$. We need to show that
$\int_X (-\f) T \wedge \omega<+\infty$. Let $\omega'$ be a smooth form cohomologous
to $\omega$ which vanishes in a small neighborhood of $F$. We can
find $h \geq 0$ a smooth function such that $\omega=\omega'+dd^c h$.
Now
\begin{eqnarray*}
\int_X (-\f) T \wedge \omega&=&\int_X (-\f) T \wedge \omega'+\int_X h T \wedge (-dd^c \f) \\
&\leq& \int_X (-\f) T \wedge \omega'+\int_X h T \wedge \alpha<+\infty,
\end{eqnarray*}
since $\f$ is bounded on the support of $T \wedge \omega'$. We have used here
that $h T \geq 0$ while $-dd^c \f \leq \alpha$. This shows that
$\f \in L^1(T)$.  Thus the measure
$\mu:=S \wedge T=(\alpha+dd^c \f) \wedge T$ is well defined.

Assume now that $\f \geq g:=A \log \dist(\cdot,F)-A$ for some contant $A>0$
and $\nu(T,p)=0$ for all $p \in F$.  Assume $g \in PSH(X,\alpha)$.  Observe that $g \in L^1(T)$. 
Hence $\theta=[\alpha+dd^c g] \wedge T$ is a well defined positive measure
which looks, locally near each point in $F$, like the projective mass of 
$T$. Since $\nu(T,p)=0$ when $p  \in F$, we infer that $\theta(p)=0$. Therefore
$\theta(\set{g=-\infty})=0$, and there exists $\chi \in {\mathcal W}$ such that
$g \in \E_{\chi}(T,\alpha)$. It now follows from Corollary \ref{cor:ineqfdtal} that $\f \in \E_{\chi}(T,\alpha)$ also, since $\f$ is less singular than $g$.

When $g$ is not $\alpha$-psh, it is $\om$-psh and hence 
$(\om+\alpha)$-psh, for some K\"ahler form $\om$.
Observe that $\varphi$ is also $(\om+\alpha)$-psh. We claim that
$\varphi \in \E(T,\alpha)$ if and only if $\varphi \in \E(T,\alpha+\omega)$.
Indeed
$$
\mu(\f,T,\a+\{ \omega \})
=\lim_{j \rightarrow +\infty} \left[ \mu_j(\f,T,\a)+{\bf 1}_{\{ \f>-j\}} \omega \wedge T \right]
=\mu(\f,T,\a)+{\bf 1}_{\{ \f>-\infty\}} \omega \wedge T.
$$
Now $\omega \wedge T(\f=-\infty)=0$ since $(\f=-\infty) \subset F$ is finite.
Thus
$$
\m\left(\mu(\f,T,\a+\{ \omega \})\right) =\{ T \} \cdot (\a+\{\omega\})
\text{ if and only if }
\m\left(\mu(\f,T,\a)\right) =\{ T \} \cdot \a,
$$ which was the desired result.
We can now conclude the proof by replacing $\alpha$ by $\alpha+\omega$ in the above argument.
\end{proof}

\section{The canonical invariant measure}\label{sec:measure}

Now let us return to the dynamical situation described in the introduction.  
For the remainder of this paper, 
$f:X \rightarrow X$ is a meromorphic transformation of a compact K\"ahler surface $(X,\om)$. We always assume that
\begin{itemize}
\item[-] $f$ is 1-stable, i.e. $(f^n)^*=(f^*)^n$ on $H^{1,1}(X,\R)$ for all $n \in \N$;
\item[-] the dynamical degrees of $f$ satisfy $1 \leq \lambda_2(f)<\l:=\lambda_1(f)$.
\end{itemize}

With these conditions, our work in \cite{ddg1} shows that the canonical current 
$T^+ = \lambda_1^{-1} f^*T^+$ in \eqref{invcurrents} exists and can be alternatively expressed
\begin{equation}
\label{T+}
T^+ = \lim_{n\to\infty}\lambda_1^{-n}f^{n*}\omega^+ = \om^+ + dd^c G^+, \text{ with } G^+ = \sum_{n=0}^\infty \frac{\Gamma^+\circ f^n}{\lambda_1^n},
\end{equation}
where $\om^+$ is a positive closed current with bounded potentials cohomologous to $T^+$ and 
$dd^c\Gamma^+ = \lambda_1^{-1}f^*\omega^+ - \omega^+$.  The canonical forward invariant current $T^- = \lambda_1^{-1} f_* T^-$ also exists and admits
a similar description 
\begin{equation}
\label{T-}
T^- = \lim_{n\to\infty} \lambda_1^{-n} f^n_*\om^- = \om^- + dd^c G^-, \text{ with } G^- = \sum_{n=0}^\infty \lambda_1^{-n} f^n_* \Gamma^-.
\end{equation}

\subsection{The dynamical energy}

We can  assume, without loss of generality, that 
$
\{ \om^+ \} \cdot \{ \om^- \}=\{T^+\} \cdot \{ T^- \}=1.
$
Our aim here is to use the techniques from \S \ref{sec:pluripot} to define a probability measure $\mu_f = T^+\wedge T^-$ and understand its geometric and dynamical properties.  

\begin{defi}
We say that $f$ has finite dynamical energy if
$$
T^+ \in \E(T^-) \text{ and } T^- \in \E(T^+).
$$
\end{defi}

We assume henceforth that our maps always have finite dynamical energy.  
Finite dynamical energy is equivalent via Proposition \ref{prop:union} to having $\chi \in \W$ such that
$$
G^+ \in \E_{\chi}(T^-,\om^+) \text{ and } G^- \in \E_{\chi} (T^+,\om^-).
$$
We say then more specifically that $f$ satisfies condition $(E_{\chi})$.  
We shall see in \S \ref{sec:examples} some weights $\chi \in \W$ that actually 
arise in specific families of examples.

\medskip

The existence of $\mu_f$, in the sense of Theorem \ref{thm:cvma}, 
is a consequence of finite dynamical energy.  When $kod(X)=0$,
the invariant current $T^-$ has bounded potentials, hence it does
not charge pluripolar sets. When $X$ is rational, we instead consider
$T^+$.

\begin{pro} \label{pro:alternativeT+}
Assume $X$ is rational and $f$ has finite dynamical energy.
Then $T^+$ does not charge pluripolar sets.
\end{pro}

\begin{proof}
The proof consists in establishing the following more precise
alternative:
\begin{itemize}
\item[-] either $T^+$ does not charge pluripolar sets,
\item[-] or it is supported on a pluripolar set, hence $f$ cannot have finite 
dynamical energy.
\end{itemize}
Decompose $T^+$ as $T^+=T^+_{np}+T^+_{pp}$, where $T^+_{np}={\bf 1}_{\set{G^+>-\infty}} T^+$ does not charge pluripolar sets,
while $T^+_{pp}={\bf 1}_{(G^+=-\infty)}$ gives full mass the pluripolar set
${G^+=-\infty}$.
Because $T^+$ is a strongly approximable laminar current (see \cite{ddg1})
It follows from \cite[Theorem 6.8]{structure} 
that this decomposition is closed, i.e. $T^+_{pp},T^+_{np}$ are closed currents. 

Since $T^+$ is invariant (under $f^*/\l$) and does not charge (critical) curves
(see Theorem 2.4 in \cite{ddg1}),
we also infer that $T^+_{pp},T^+_{np}$ are both invariant. Now $T^+$ is an extremal
point of the cone of positive closed invariant currents
(see Remark 2.2 in \cite{ddg1}). Thus it follows that either
$T^+=T^+_{np}$ does not charge pluripolar sets (see Proposition \ref{pro:plp2}),
or $T^+=T^+_{pp}$ is supported on the pluripolar set $\set{G^+=-\infty}$.
In the latter case, it follows from Proposition \ref{pro:degen} that
$f$ cannot have finite dynamical energy.
\end{proof}

As it will be seen below and in \cite{ddg3}, the finite dynamical energy
condition will  allow us to understand 
the dynamics of $\mu_f$ quite thoroughly.  
>From Proposition \ref{pro:alternative} we get the following nice alternative, which 
emphasizes the naturality of  this assumption.

\begin{thm}\label{thm:alternative}
Assume that $T^+ \in L^1(T^-)$, so that $\mu_f=T^+ \wedge T^-$ is well defined. 
Then
\begin{itemize}
\item[-] either $\mu_f$ charges the pluripolar set $\set{G^+=-\infty} \cup \{G^-=-\infty\}$,
\item[-] or $f$ has finite dynamical energy.
\end{itemize}
\end{thm}

\subsection{Mixing}

\begin{thm} \label{thm:mixing}
Assume $f$ has finite dynamical energy. Then the measure 
$\mu_f=T^+ \wedge T^-$ is $f$-invariant and mixing, 
and it does not charge pluripolar sets.
\end{thm}

\begin{proof}
It follows from Proposition \ref{pro:alternativeT+} and Theorem \ref{thm:cvma}
that $\mu_f=T^+ \wedge T^-$ is a well defined probability measure
which does not charge pluripolar sets.
In particular, $\mu_f=T^+ \wedge T^-$ does not charge the pluripolar set 
$I^+ \subset \set{G^+=-\infty}$.  Moreover by Theorem \ref{thm:cvma},
$$
\mu_f=\lim_{n,p \rightarrow +\infty} \mu_{n,p},
\text{ where }
\mu_{n,p}:=\frac{1}{\l^n}(f^n)^* \om^+ \wedge \frac{1}{\l^p}(f^p)_* \om^-.
$$
Since $f_* \mu_{n,p}=\mu_{n-1,p+1}$ and since the operator $f_*$ is continuous
on the set of probability measures which do not charge the indeterminacy set $I^+$,
we infer $f_*\mu_f=\mu_f$, i.e. $\mu_f$ is an invariant probability measure.

We now show that $\mu_f$ is mixing. Let $h$ and $k$ be test functions on $X$. We need to show
that
$$
\int_X k \,  h \circ f^n \, d\mu_f \longrightarrow \int_X h \, d\mu_f
\int_X k \, d\mu_f.
$$

Set $\mu_j:=[\om^-+dd^c v_j] \wedge T^+$, where $v_j:=\max(G^-,-j)$ are the canonical 
approximants of $G^-$. Fix $\chi \in {\mathcal W}$ such that
$G^- \in \E_{\chi}(T^+,\om^-)$.
It follows from Corollary \ref{cor:ineqfdtal}
that $\chi \circ G^- \in L^1(\mu_f)$, 
$\chi \circ v_j \in L^1(\mu_j)$ uniformly in $j$, and also
that $\mu_j \equiv \mu_f$ in the plurifine open set
$\{G^->-j\}$. Therefore 
$$
\left| \int_X k \,  h \circ f^n \, d\mu_f 
-\int_X k \,  h \circ f^n \, d\mu_j \right| 
 \leq ||k||_{L^{\infty}} ||h||_{L^{\infty}}(\mu_f + \mu_j)\{G^-\leq -j\}
$$
converges to zero as $j \rightarrow +\infty$, uniformly with respect to $n \in \N$.

It suffices then to replace $\mu_f$ by $(\om^-+dd^c v) \wedge T^+$, where $v$ is a {\it bounded}
$\om^-$-psh function and show
$$\int_{X} k\, h\circ f^{n}(\om^-+dd^c v) \wedge T^+ \rightarrow \left(\int k (\om^-+dd^c v) \wedge T^+\right)\left(\int h\,d\mu_f \right).$$
Observe that here we can replace $\om^-$
by $\theta^-$, which is smooth.  So it suffices  to 
show the convergence for an arbitrary K\"ahler form $\om$,
instead of $\theta^-$.

If $v$ is smooth the convergence follows from Proposition \ref{pro:extrem} below. For the general case, recall that
we can approximate $v$ from above by a decreasing family of smooth $\om$-psh 
functions
$v_{\e}$. The proof will be finished if we show that
$\int_X k h \circ f^n \, dd^c(v_{\e}-v) \wedge T^+$ converges to zero
as $\e \rightarrow 0$, uniformly with respect to $n$. For this we first  integrate
by parts 
\begin{eqnarray*}
\left| \int_X k\, h \circ f^n \, dd^c(v_{\e}-v) \wedge T^+ \right| \leq
\left| \int_X d k \wedge d^c(v_{\e}-v) \wedge h \circ f^n \, T^+ \right| \\
+\left| \int_X d h \circ f^n \wedge d^c(v_{\e}-v) \wedge k \, T^+ \right|
=I(n,\e)+II(n,\e),
\end{eqnarray*}
and use the Cauchy-Schwarz inequality. For the first term we get 
$$
I(n,\e) \leq ||h||_{L^{\infty}} ||k||_{{\mathcal C}^1}
\left[ \int d(v_{\e}-v) \wedge d^c (v_{\e}-v) \wedge T^+ \right]^{1/2}
\rightarrow 0,
$$
because $v$ is bounded.

Treating the second integral $II(n,\e)$ similarly gives rise to an integral of the form 
$$
\int d(h\circ f^n)\wedge d^c(h\circ f^n)\wedge  k T^+.
$$ 
Then we argue as in the proof of \cite[Theorem 3.3]{ddg1} to get that this integral is $O(\lambda_2^n/\lambda_1^n)$.
\end{proof}

\begin{pro} \label{pro:extrem}
Let $h$ be a (smooth) test function on $X$. Then
$$
h \circ f^n\, T^+ \longrightarrow c T^+,
\text{ where } c=\int h\, d\mu_f
$$
\end{pro}

\noindent Again, the proof is similar to that of \cite[Theorem 3.3]{ddg1}, so we omit it. 
%


%

\subsection{Geometric intersection}\label{subs:isect}

In this paragraph, $X$ is supposed to be projective.
We showed in \cite{ddg1} that the invariant currents $T^+$ and $T^-$ admit important 
geometric structures. 
More precisely $T^+$ (resp. $T^-$) is a \emph{strongly approximable} \emph{laminar} (resp. \emph{woven}) current.

Our purpose here is to show that the measure $\mu_f$ can be alternatively 
obtained by ``intersecting'' these geometric structures. Recall that another result that has been obtained using laminarity 
is the fact that $T^+$ --and consequently $\mu_f$-- does not charge pluripolar sets. 

We first very briefly recall some preliminaries on geometric currents (see \cite{structure, birat} and also \cite{ddg1, ddg3} for more details). The construction of these geometric structures requires subdivisions $\qq$ of $X$ into families of ``cubes'' $Q\in \qq$, which are obtained by projecting along two generically  transverse linear pencils, and taking intersections of subdivision of the projection bases into squares. In particular we have some freedom in the choice of the projections and the squares. The spaces parametrizing  projections and squares are manifolds so we can speak of a 
 ``generic subdivision'' where genericity is understood in the sense of Lebesgue measure.

That $T^+$ is a strongly approximable laminar current  means that if $\qq$ is a generic 
increasing family (actually,  a sequence) of subdivisions by cubes, then 
$T^+$ is the limit of an increasing family $$T_\qq^+ = \sum_{Q\in \qq} T^+_Q,$$  the current $T^+_Q$ is uniformly laminar in $Q$.
Furthermore,   $T^+_Q$ can be written as an integral over a measured family of
disjoint submanifolds  of uniformly  bounded volume
\begin{equation}\label{eq:integral}
T^+_Q = \int_{\alpha\in A_Q^+} [\Delta^+_\alpha] \, d\nu_Q^+(\alpha).
\end{equation}
 We  have the important estimate 
\begin{equation}\label{eq:mass}
\m(T^+-T^+_\qq)\leq C r^2.
\end{equation}
The same holds for $T^-$, with ``woven'' instead of ``laminar''. In this case  
\eqref{eq:integral}
becomes 
\begin{equation}\label{eq:integral-}
T^-_Q = \int_{\alpha\in A_Q^-} [\Delta^-_\alpha] \, d\nu_Q^-(\alpha), 
\end{equation}
where the $\Delta^-$ are allowed to intersect and can be singular.
In view of \eqref{eq:integral} and \eqref{eq:integral-}
we can naturally define the {\em geometric intersection} of $T_Q^+$ and $T_Q^-$ 
as 
$$
T^+_Q\geom T^-_Q = \int_{A_Q^+}\int_{A_Q^-} [\Delta^+_\alpha \cap \Delta^-_{\alpha'}]\, d\nu_Q^+(\alpha)\,d\nu_Q^-(\alpha'),
$$
where by definition, $[\Delta_\alpha \cap \Delta_{\alpha'}]$ is the sum of Dirac masses at isolated intersections, counting multiplicities.

Assume now that $T^+\wedge T^-$ is well defined, in the $L^1$ or energetic sense. 
We say that the wedge product $T^+\wedge T^-$ is a {\em geometric intersection} if the 
family of measures $T^+_\qq\geom T^-_\qq := \sum_{Q\in \qq} T^+_Q\geom T^-_Q$ increases to $T^+\wedge T^-$ when $\qq$ is any family of generic subdivisions into cubes of size $r\cv 0$. 


The following basic result asserts that for uniformly geometric currents the wedge product is geometric.

\begin{pro} \label{pro:geom} 
With notation as above,
assume that  $T_\qq^+ \in L^1_{loc}(T^-_\qq)$. Then the wedge product
$T^+_\qq \wedge T^-_\qq$ is geometric, that is, 
$$T^+_\qq\wedge T^-_\qq = \sum_{Q\in \qq}T^+_Q\geom T^-_Q.$$
\end{pro}

\begin{proof}
The proposition follows by applying  Lemma \ref{lem:integral} below twice. 
Notice that if $\Delta$ and $\Delta'$ are submanifolds  in $Q$, $[\Delta] \in L^1_{loc}([\Delta'])$ iff $\Delta$ and 
$\Delta'$ only have isolated intersections. If this holds, then $[\Delta] \wedge [\Delta']
=[\Delta\cap\Delta']$ (see \cite{demailly}).
\end{proof}

\begin{lem}\label{lem:integral}
Assume that $R$ and $S$ are two positive closed currents in an open set $Q$, such that 
$S\in L^1_{\rm loc}(R)$ and $S$ admits a decomposition 
$S = \int S_\alpha d\nu(\alpha)$  as an integral of positive closed currents. Then for $\nu$ a.e. $\alpha$, $S_\alpha\in L^1_{\rm loc}(R)$ and we have the decomposition 
$$ R \wedge S = \int (R\wedge S_\alpha) d\nu(\alpha).$$ 
\end{lem}

\begin{proof}
The result is local so consider a small ball $B\Subset Q$.
 By definition, $\alpha\mapsto \m(S_\alpha\rest{B})$ is locally $\nu$ integrable. 
 By Lemma \ref{lem:potential} below, for every $\alpha$ there exists a non positive psh function  $u_\alpha$  in $B$  such that $dd^cu_\alpha=S_\alpha$ 
and  $\norm{u_\alpha}_{L^1(B')}\leq C\m(S_\alpha\rest{B})$, where $B'$ is a slightly smaller ball. Hence  $u=\int u_\alpha d\nu(\alpha)$ 
is a well defined psh function which is a potential for $S$ in  $B'$. 

Now since $u\in L^1_{\rm loc} (R)$, we get that for a.e. $\alpha$, 
$u_\alpha\in L^1_{\rm loc} (R)$, which is the first assertion of the lemma, while the second follows by applying $dd^c$ to the formula $uR= \int u_\alpha R d\nu(\alpha)$. 
\end{proof}

The following lemma is classical and goes back at least to \cite{lelong} (see \cite{BEM} for a brief treatment).

\begin{lem}\label{lem:potential}
Let $T$ be a positive closed current with finite mass in the unit ball $B\subset \cd$. Then $T$ admits  potential $u$ which is canonical, negative on $B(0, 1/2)$ and satisfies $\norm{u}_{L^1(B(0, 3/4)}\leq C \m(T)$, with $C$ a universal constant.
\end{lem}

%
%
%

We now arrive to the main result in this paragraph.

\begin{thm}
Assume $X$ is projective and that $f$ has finite dynamical energy.
Then the wedge product $T^+ \wedge T^-$ is geometric.
\end{thm}

\begin{proof}  Let us first assume for simplicity that $T^+ \in L^1(T^-)$.
This guarantees that all the wedge products $T^{\pm} \wedge T^{\mp}_\qq$,
$T^+_\qq \wedge T^-_\qq$ are well defined in the $L^1_{\rm loc}$ 
sense  and bounded from
above by $T^+ \wedge T^-$. The wedge product
$T_\qq^+ \wedge T_\qq^-$ is geometric by Proposition \ref{pro:geom} 
and we need to show
that this product increases towards $\mu_f=T^+ \wedge T^-$, as 
$r>0$ decreases towards $0$.
Here the situation is symmetric so it suffices to estimate the mass
$\m(T^- \wedge (T^+-T_\qq^+))$. We follow the proof and notation of Theorem 5.2 in \cite{birat}.

Shifting the cubes slightly, we can assume the mass of $\mu_f$ is not concentrated
near the boundary of ${\mathcal Q}$.
Let $\p_\qq=\sum_{Q \in {\mathcal Q}} \p_Q$ be a test function 
such that each function $\p_Q$ is supported in the cube $Q$,
satisfies $0 \leq \p_Q \leq 1$ and is identically 
equal to $1$ in the major part of $Q$. It suffices to show
that $\int_X \p_\qq T^- \wedge (T^+-T_\qq^+) \rightarrow 0$.
Note that $\p$ can be chosen so that 
$||\nabla \p_\qq||=O(1/r)$ and $||dd^c \p_\qq||=O(1/r^2)$.

Let $G_j^- =\max\{G^-,-j\}$, $T_j^- = \om^- + dd^c G_j^-$ and $\mu_j = T^-_j\wedge T^+$ as before.  Then
\begin{equation}
\label{split}
\int \p_\qq T^- \wedge (T^+-T_\qq^+) = \int \p_\qq T^-_j\wedge(T^+ - T_\qq^+)
+ \int \p_\qq (T^- - T_j^-)\wedge(T^+ - T_\qq^+).
\end{equation}
In the second integral on the right, we use the facts that $T^+ \geq T_\qq^+$ are closed on
$\supp\psi_\qq$ and that $T^-\wedge (T^+-T_\qq^+) =T_j^-\wedge (T^+-T_\qq^+)$ on $\{G^->-j\}$ to estimate
\begin{eqnarray}\label{split2}
\abs{\int_X \p_\qq (T^- - T_j^-)\wedge(T^+ - T_\qq^+) }& = &\abs{
\int_{\{G^-\leq -j\}} \p_\qq (T^- - T_j^-)\wedge(T^+ - T_\qq^+)} \\ \notag
& \leq & \int_{\{G^-\leq -j\}}\psi_\qq(\mu_f + \mu_j),
\end{eqnarray}
which tends to zero as $j\to \infty$ uniformly in $r$ (recall that $\mu_j(\set{G^-\leq -j}) = \mu(\set{G^-\leq -j})$ see \eqref{eq:mumass}). 
>From now on we fix $j$ such that this integral  is small.

It remains to control the first integral on the right hand  side of \eqref{split}.
For convenience it is better to write $T_j^- = \theta^- +dd^c H_j$, where $\theta^-$ is a smooth form cohomologous to $\omega^-$ and $H_j$ is still bounded. 
 This gives
$$
\int \p_\qq T^-_j\wedge(T^+ - T_\qq^+) = -\int d\psi_\qq\wedge d^cH_j\wedge (T^+ - T_\qq^+) + 
\int \psi_\qq \theta^-\wedge(T^+ - T_\qq^+).
$$
The second term on the right tends to zero as $r\to 0$ by laminarity of $T^+$.  The first we estimate with Schwarz' inequality and the estimate given by the ``strong approximability'' of $T^+$.
$$
\left|\int d\psi_\qq\wedge d^cH_j\wedge (T^+ - T_\qq^+)\right|^2
\leq \norm{d \psi_\qq}^2_{\infty} O(r^2) \int dH_j\wedge d^c H_j\wedge (T^+ - T_\qq^+).
$$
Since $\norm{d \psi_\qq}^2_{\infty} = O(r^{-2})$, it suffices to know that the last integral tends to zero with $r$.
But this happens because the bounded function $H_j$ has gradient in $L^2$ with respect to $T^+$, and because $T_\qq^+\nearrow T^+$ (see the end of the argument in \cite{birat}).

\medskip

 We now no longer assume that $T^+ \in L^1(T^-)$.
One difficulty is that 
 we have to justify the existence of the local wedge products  $T_Q^+ \wedge T^-$, etc. This is ensured by the following lemma (which, it is perhaps worth stressing, does not require homogeneity of $X$).

\begin{lem}
There exists $S_\qq^+$ arbitrary close to $T_\qq^+$ such that  $G^- \in L_{loc}^1(S_\qq^+)$. 
Likewise, there   exists $S_\qq^-$ arbitrary close to $T_\qq^-$ such that  $G^+ \in L_{loc}^1(S_\qq^-)$.
\end{lem}

\begin{proof}
Since $T^- \in {\mathcal E}(T^+)$, by the first item of Proposition \ref{pro:plp}
 $G^-$ is not $-\infty$ a.e. on $T^+$. Thus, given a cube $Q$,  $G^-$ is not identically $-\infty$ on  $T_Q^+$. Now
$
T^+_Q=\int [\Delta_\a] d\nu_Q^+(\a),
$
hence for $\nu^+$-a.e. disk $\Delta_\a$,
$\Phi(\a):=\int_{\Delta_\a} G^-$ is finite. Consider now 
$$
\nu^+_N:=\nu^+\rest{\set{\Phi>-N}}
\text{ and }
T_{Q,N}^+:=\int [\Delta_\a] d\nu^+_N(\a).
$$
Then by construction $G^- \in L^1(T^+_{Q,N})$ and $T^+_{Q,N}$ is arbitrarily close to
$T_Q^+$. 

The argument for  $S_\qq^-$ is similar. 
\end{proof}

>From now on we replace the $T_\qq^\pm$ with the $S_\qq^\pm$ as given by the previous lemma, so  all the local wedge products ($T_Q^+\wedge T_Q^-$, $T_Q^+\wedge T^-$, 
$T^+\wedge T_Q^-$)  are well defined in the usual $L^1$ sense.
We further assume that $S_\qq^\pm$ is so close to $T_\qq^\pm$ that the estimate \eqref{eq:mass} is satisfied.

We now  need to justify the inequalities $T_\qq^+\wedge T^-_\qq\leq T^+\wedge T^-$, etc. This follows from the following simple observation: let 
$T^-_j = \omega^- + dd^c\max(G^-, -j)$. We have that $T^+\wedge T^-_j\cv T^+\wedge T^-$ (by the energy approach) and  
$T^+_\qq\wedge T^-_j\cv T^+_\qq\wedge T^-$ (by the classical approach) as $j\cv\infty$. Since $T^+_\qq\wedge T^-_j\leq T^+\wedge T^-_j$ we conclude that $T^+_\qq\wedge T^-\leq T^+\wedge T^-$. 
In the same way we obtain that $T^+\wedge T^-_\qq \leq T^+\wedge T^-$. 
The inequality $T^+_\qq\wedge T^-_\qq\leq T^+_\qq\wedge T^-$ is obvious since all wedge products are defined in the classical sense.

Starting from here the proof is identical to the $L^1$ case. \end{proof}

\section{Homogeneous weights}\label{sec:homogeneous}

In this section we give criteria allowing to verify in practice (see \S\ref{subs:irrational}) 
the finite energy condition for the 
{\em homogeneous weights} $\chi(t) = -(-t)^p$, with $0<p\leq 1$.

\medskip

Recall from \eqref{T+} and \eqref{T-} that we can write  $T^\pm = \omega^\pm+dd^cG^\pm$. 
To analyze the potential $G^+$, it is easier to  use the function $\Gamma^+$ as an intermediary.  
Indeed, Proposition 2.4 in \cite{ddg1} implies that 
$\Gamma^+ \geq A\log\dist(\cdot,I^+)-B$ for some $A,B>0$, where $I^+\subset X$ 
is the indeterminacy set of the map $f$. 
Let us call an indeterminacy point 
$$
p\in I^+ \text{ \emph{spurious} if } f(p)\cdot\omega^+ = 0. 
$$
At every non-spurious point, 
we also have the reverse inequality $\Gamma^+ \leq A\log\dist(\cdot, p)-B$. 

Similar inequalities hold for $\Gamma^-$.  
We let $E^+$ denote the \emph{exceptional set} of $f$, i.e. the union of those curves collapsed 
by $f$ to points, and set $I^- = f(E^+)$.  Then $\Gamma^- \geq A\log\dist(\cdot,I^-) - B$ 
with the reverse inequality holding if and only $f^{-1}(p)\cdot \omega^- \neq 0$ for every $p\in I^-$.
Likewise, we say that a point in 
$$
p \in I^- \text{ is {\em spurious} if } f^{-1}(p)\cdot \omega^- = 0.
$$
We refer the reader to \cite{ddg1} for more about spurious points.  Here we point out only that: 
\begin{enumerate}
\item if the invariant cohomology classes $\{T^+\}$, $\{T^-\}$ are K\"ahler, then
there is no spurious point of indeterminacy in resp. $I^+$, $I^-$;
\item when $f$ is bimeromorphic, we can always perform a bimeromorphic
change of coordinates to get rid of spurious indeterminacy \cite[Prop. 4.1]{bd}.
\end{enumerate}
It is plausible that a similar result holds when the topological degree
$\lambda_2(f)$ is merely smaller than the first dynamical degree $\l=\lambda_1(f)$.
We refer the reader to \S 5 in \cite{ddg1} for some results in
this direction. 

\medskip

The Lelong numbers of $T^+$ vanish on $I^-$ (see \cite[Theorem 2.4]{ddg1}).
Hence Proposition \ref{pro:energycriter} gives

\begin{pro}
We have $\Gamma^- \in \E(T^+,\omega^-)$.
\end{pro}

A similar result holds for $\Gamma^+ / T^-$ if we know that $T^-$ has zero Lelong
number at each point in $I^+$ (which we do e.g. when $\lambda_2(f)=1$).
\medskip

The main theme in this section is that for homogeneous weights, it is possible to pass 
from control on energy of $\Gamma^\pm$ to control on that of $G^\pm$. This idea originates in \cite{bd}.
We set
$$
\E^p(T,\om):=\E_{\chi}(T,\om)
\text{ and } \nabla^p(T,\om):=\nabla_{\chi}(T,\om)
$$
where
$\chi(t)=-(-t)^p$, $ 0 <p \leq 1$.

\begin{pro} \label{pro:homograd}
Fix $0<p \leq 1$. 
Then 
$$
\Gamma^+ \in {\mathcal E}^p(T^-,\omega^+)
\text{ if and only if }
G^+ \in \nabla^p(T^-,\omega^+).
$$ 

If $\l^p>\lambda_2(f)$ then we similarly have that
$$
\Gamma^- \in {\mathcal E}^p(T^+,\omega^-)
\text{ if and only if }
G^- \in \nabla^p(T^+,\om^-).
$$
\end{pro}

\begin{proof} 
Since $\Gamma^+$ and $G^+$ are both $\omega^+$-psh functions such that 
$G^+ \leq \Gamma^+ \leq 0$, it follows from Corollary \ref{cor:ineqfdtal} that
$$
G^+ \in {\mathcal E}_{\chi}(T^-, \om^+) \Longrightarrow \Gamma^+ \in {\mathcal E}_{\chi}(T^-, \om^+),
$$
for any weight $\chi \in {\mathcal W}$.

Assume conversely now that $\Gamma^+ \in {\mathcal E}_{\chi}(T^-, \om^+)$ for the special homogeneous weight $\chi(t)=-(-t)^p$, $0 <p \leq 1$.  Since $\chi'\circ G^+ \leq \chi'(\l^{-j} \Gamma^+ \circ f^j)
=\l^{-j(p-1)} \chi'(\Gamma^+ \circ f^j)$, it follows from the Cauchy-Schwarz inequality
and $(f^j)_* T^-=\l^j T^-$ that
\begin{eqnarray*}
\lefteqn{
\left( \int_X \chi' \circ G^+ \, d G^+\wedge d^c G^+ \wedge T^- \right)^{1/2}  } \\
&& \leq \sum_{j \geq 0} \frac{1}{\l^{j}} 
\left( \int_X \chi' \circ G^+ d(\Gamma^+ \circ f^j) \wedge d^c (\Gamma^+\circ f^j) 
\wedge T^- \right)^{1/2} \\
&&\leq \sum_{j \geq 0} \frac{1}{\l^{jp/2}} 
\left( \int_X \chi' \circ \Gamma^+ d\Gamma^+ \wedge d^c \Gamma^+ \wedge T^- \right)^{1/2}<+\infty.
\end{eqnarray*}
Hence $G^+ \in \nabla_{\chi}(T^-,\om^+)$.

For $\Gamma^-$ and $G^-$, the proof is very similar except that in the estimate analogous to the one in the previous display, pushforward does not distribute over products as well as pullback. If $\eta$ is a $(1,0)$ form, we have only an 
inequality
$$
i\, f_* \eta \wedge \overline{f_* \eta} \leq \lambda_2(f) \cdot  f_*(i \, \eta \wedge \overline{\eta}).
$$
Arguing as above and setting 
$E_{\chi}(\Gamma^-)=\int \chi' \circ \Gamma^- d\Gamma^- \wedge d^c \Gamma^- \wedge T^+$, we thus get
$$
\left( \int_X \chi' \circ G^- \, d G^-\wedge d^c G^- \wedge T^+ \right)^{1/2} 
\leq \sum_{j \geq 0} \left(\frac{\lambda_2}{\l^p} \right)^{j/2} 
\left[E_{\chi}(\Gamma^-)\right]^{1/2},
$$
which is finite if $\l^p>\lambda_2$.
\end{proof}

Since $L^1\subset L^p$ for $0<p\leq 1$, Propositions \ref{pro:grad} and \ref{pro:homograd} directly imply

\begin{cor}  Suppose $0<p\leq 1$ and $T^+\in L^1(T^-)$.  Then $\Gamma^+$ belongs to $\E^p(T^-,\omega^+)$ if and only if $G^+$ does.  If in addition $\lambda_1^p>\lambda_2$, then $\Gamma^-$ belongs to $\E^p(T^+,\omega^-)$ if and only if $G^-$ does. 
\end{cor}

Having largely reduced the problem of controlling $G^\pm$ to that of controlling $\Gamma^\pm$, we now seek effective means of accomplishing the latter.  The computations below will be the same for $\Gamma^+$ and $\Gamma^-$, so we work only with $\Gamma^+$.  We let $\Omega = \omega^+ + c\omega$. For $c>0$ large enough, both $\Gamma^+$ and $\varphi := \log\dist(\cdot,I^+)$ are $\Omega$-psh functions.  Since $\Gamma^+ \geq A\varphi - B$, Lemma \ref{lem:ineqfdtal} gives

\begin{pro} 
\label{pro:logdist}
Let $\chi$ be any weight function.  Then $\varphi\in \E_\chi(T^-,\Omega)$ implies that $\Gamma^+\in \E_\chi(T^-,\Omega)$.
\end{pro}

Hence we consider weighted energy of $\varphi$.  Let $\varphi_j$ be the canonical approximants.  Then up to finite additive constants we have
\begin{eqnarray*}
0 & \leq & -\int \chi\circ\varphi_j (\Omega + dd^c\varphi_j) \wedge T^- = -\int -G^- dd^c\varphi_j \wedge dd^c\chi\circ\varphi_j \\
& = & \int -G^- (\chi''\circ\varphi_j)\,d\varphi_j\wedge d^c\varphi_j\wedge dd^c\varphi_j + \int -G^-(\chi'\circ\varphi_j)\,(dd^c\varphi_j)^2 \\
& \leq & \int_{X\setminus I^+} -G^- \chi''\circ\varphi\,d\varphi\wedge d^c\varphi\wedge dd^c\varphi + \int_{X\setminus I^+} -G^-\chi'\circ\varphi\,(dd^c\varphi)^2
+ \int -G^- (\chi'\circ\varphi) \nu_j \\
& = & I + II + III,
\end{eqnarray*}
where $\nu_j = (dd^c\varphi_j)^2|_{\varphi = -j}$ is a positive measure.  Note also that $d\varphi_\wedge d^c\varphi_j\wedge dd^c\varphi_j^2$ puts no mass on $\varphi = -j$.  One can verify this by replacing
$\max\{\cdot,j\}$ with a smooth convex approximation in the definition of $\varphi_j$ and then computing directly.

To bound integral I, we compute in local coordinates that
$$
d\varphi\wedge d^c\varphi\wedge dd^c\varphi \leq \frac{C\,dV(x)}{\dist(x,I^+)^4}
$$
So if $m^-(t)$ is the spherical mean of $G^-$ on the set $\{\dist(p,I^+) = e^t \}$, we get
$$
I \leq C\int_{-\infty}^{\max \varphi} -m^-(t)\chi''(t)\,dt, 
$$
We claim that II is always finite.  To see this, let $\pi:\hat X\to X$ be the blowup of $X$ along the finite set $I^+$.  Then direct computation in local coordinates about $I^+$ reveals that on $X\setminus I^+$, one has $(dd^c\varphi\circ \pi)^2\leq dV$, where $dV$ is a smooth volume form on $\hat X$.  Since $G^-\circ\pi$ is $\pi^*\alpha^-$-psh on $\hat X$, this gives us that
$$
II = \int_{\hat X\setminus\pi^{-1}(I^+)} -(G^-\circ\pi)(\chi'\circ\varphi\circ\pi)\,(dd^c\varphi\circ\pi)^2 \leq 
     \int_{\hat X\setminus\pi^{-1}(I^+)} -G^-\circ\pi\,dV < \infty.
$$
Finally, to deal with III, we note that for $j$ large $\nu_j$ is uniformly (in $j$) proportional to normalized spherical measure on $\{\dist(p,I^+) = e^{-j}\}$.  Hence we obtain that
$$
III \leq -C m^-(-j)\chi'(-j).
$$
With these estimates we arrive at two conclusions very much in the spirit of \cite{bd}.  We state them only for $\Gamma^+$, but the analogous assertions for $\Gamma^-$ are equally valid and proved in the same way.

\begin{cor} 
\label{cor:e1}
If $G^-$ is finite at all points in $I^+$, then $\Gamma^+ \in \E^1(T^-,\om^+)$.  Hence also $G^+ \in \nabla^1(T^-,\om^+)$.  If moreover $T^+\in L^1(T^-)$ then $G^+\in \mathcal{E}^1(T^-, \omega^+)$.
\end{cor}

\begin{proof}
By Propositions \ref{pro:logdist} and \ref{pro:homograd}, it suffices to verify that the above bounds on I and III are finite (and uniform as $j\to\infty$).  In this case, we have $\chi(t) = t$, and in particular $\chi'' \equiv 0$.  Hence I is trivially finite.

On the other hand, $\chi'\equiv 1$ and as $t\to -\infty$, $m^-(t)$ decreases to the average value of $G^-$ on the set $I^+$.  In particular the bound on III is uniform as $j\to\infty$ if $G^-$ is finite on $I^+$. 
\end{proof}

For other homogeneous weight functions, our estimates on I and III immediately give an analogous criterion.

\begin{cor}
\label{cor:ep}
Let (as above) $m^-(t)$ denote the mean of $G^-$ on $\{\dist(x,I^+)=e^t\}$.  Suppose for some $q\in (0,1)$ that 
$$
\limsup_{t\to-\infty} |t|^{q-1}|m^-(t)|<\infty.
$$  
Then for any $0<p< q$, we have $\Gamma^+\in \E^p(T^-,\om^+)$.  Hence $G^+\in \nabla^p(T^-,\om^+)$.  If additionally $T^+\in L^1(T^-)$, then $G^+\in \E^p(T^-,\om^+)$.
\end{cor}

In some circumstances (e.g. polynomial maps of $\C^2$) it is possible to see directly that $T^+\in L^1(T^-)$.  Corollaries \ref{cor:e1} and \ref{cor:ep} are adequate by themselves for these situations.  In other circumstances, however, it is not easy to verify that $T^+\in L^1(T^-)$.  We show now that one can avoid doing this if there are no spurious points in $I^+$ or $I^-$.

\begin{thm}
\label{thm:nospurs}
Suppose that $0<p \leq 1$ is chosen so that $\lambda_1^p > \lambda_2$ 
and that $G^+ \in \nabla^p(T^-,\omega^+)$, $G^-\in\nabla^p(T^+,\omega^-)$.  
If there are no spurious points in $I^+$ or $I^-$, 
then it is further true that $G^+ \in {\mathcal E}^p(T^-,\omega)$ and 
$G^- \in {\mathcal E}^p(T^+,\omega)$.
\end{thm}

We prove the theorem in a sequence of lemmas, focusing mainly on $G^+$.  
The case of $G^-$ is identical, except that as in the proof of Proposition \ref{pro:homograd}, 
we need the condition $\lambda_1^p > \lambda_2$ to make the triangle inequality work 
in a couple of places below.

\begin{lem}
There exists $k\in\N$ and $c>0$ such that for every $n\in\N$, we have
$$
\frac{f^{nk}_*\omega}{\lambda_1^{nk}} \leq c\omega + dd^c w_n
$$
where $w_n \geq -cG^-$ is $c\omega$-psh 
\end{lem}

We remark before continuing that if the indeterminacy set of $f$ has no spurious points, 
then neither does that of $f^k$.  Therefore, there is no particular harm in assuming 
for the sake of notational simplicity that $k=1$ when we apply the lemma.

\begin{proof}
Let $\theta^-$ be a smooth form cohomologous to $T^-$.  Then by Theorem 1.3 in \cite{ddg1} 
we have $\om = a\theta^- + \eta$ for some $a>0$ and a smooth form $\eta$ such that 
$\lambda_1^{-k}f^k_*\eta$ tends to zero in cohomology.  Thus for $k$
large enough, we have
$$
\frac{f^k_*\om}{\lambda_1^k} \leq a\theta^- + \frac12\omega + dd^c w
$$
where $w$ is a quasipsh function that is smooth away from $I^-(f^k)$.  Since there are no spurious
points in $I^-$, we have in fact that $w \geq bG^-$ for some $b>0$.  
Iterating this inequality then gives
$$
\frac{f^{kn}_*\om}{\lambda_1^{kn}} \leq a b_n \theta^- 
+ \frac{1}{2^n}\omega 
+ dd^c \left( \sum_{j=0}^{n-1} \frac{f^{jk}_* w}{2^{(n-1-j)}\lambda_1^{jk}} \right)
+dd^c \left( \sum_{j=0}^{n-2} ab_{n-2-j} \frac{1}{\l^{jk}} (f^{jk})_* \Gamma^- \right),
$$
where $b_n=\sum_{\ell=0}^{n-1} 2^{-\ell} \leq 2$.
Since  
$\frac{f^j_* \Gamma^-}{\lambda_1^j} \geq \frac{f^j_* G^-}{\lambda_1^j} \geq G^-$, 
the lemma follows immediately. 
\end{proof}

\begin{lem}
Both $\int -\chi\circ G^+\, T^+\wedge\omega$ and $\int-\chi\circ G^-\, T^-\wedge\omega$ are finite.
\end{lem}

\begin{proof}
Modulo finite additive constants, the first integral is estimated by
\begin{eqnarray*}
\left( \int -\chi\circ G^+\, T^+\wedge\omega \right)^{1/2}
& = &
\left( \int \chi'\circ G^+ dG^+ \wedge d^c G^+ \wedge \omega \right)^{1/2}\\
& \leq & 
\sum_{n=1}^\infty \left( \int \chi'\left(\frac{\Gamma^+\circ f^n}{\lambda_1^n}\right) 
\frac{d(\Gamma^+\circ f^n) \wedge d^c (\Gamma^+\circ f^n)}{\lambda_1^{2n}}\wedge\omega \right)^{1/2}\\
& = &
\sum_{n=1}^\infty \frac{1}{\lambda_1^{np/2}}
\left(\int \chi'\circ\Gamma^+\,d\Gamma^+ \wedge d^c\Gamma^+\wedge\frac{f^n_*\omega}{\lambda_1^n} \right)^{1/2} \\
& \leq &
\sum_{n=1}^\infty \frac{1}{\lambda_1^{np/2}}
\left( \int \chi'\circ\Gamma^+\,d\Gamma^+ \wedge d^c\Gamma^+\wedge (c\omega + dd^c w_n) \right)^{1/2},
\end{eqnarray*}
where $c,w_n$ are as in the previous lemma.  If we take the contributions to the integral from $c\omega$ and from $dd^c w_n$ separately, then the first contribution is finite by Proposition \ref{pro:energycriter}.  The second contribution is
handled up to additive constants as follows.
\begin{eqnarray*}
\int \chi'\circ\Gamma^+\,d\Gamma^+ \wedge d^c\Gamma^+\wedge dd^c w_n & = &
\int -w_n \,dd^c(\chi\circ\Gamma^+)\wedge dd^c\Gamma^+ \\
& = & \int - w_n  \,(dd^c\chi\circ\Gamma^++\om^+)\wedge (dd^c\Gamma^++\om^+) +O(1)\\
& \leq & c\int - G^- \,(dd^c\chi\circ\Gamma^++\om^+)\wedge (dd^c\Gamma^++\om^+) +O(1) \\
& = & c\int -G^-\,dd^c(\chi\circ\Gamma^+)\wedge dd^c\Gamma^+ +O(1)\\
& = & c\int \chi'\circ\Gamma^+\,d\Gamma^+ \wedge d^c\Gamma^+\wedge T^-+O(1)
\end{eqnarray*}
which is finite by the hypotheses of Theorem \ref{thm:nospurs}
\end{proof}

\begin{lem}
Both $\int -\chi\circ G^+\,T^-\wedge \omega$ and $\int -\chi\circ G^-\,T^+\wedge \omega$ are finite.
\end{lem}

\begin{proof}
We treat the first integral only.
Let $\Omega = \omega^+ + \omega^-$.  Then from Proposition 2.5 in \cite{gz2} we have up to additive constants that
\begin{eqnarray*}
\int -\chi\circ G^+\, T^- \wedge \om 
& = & 
\int -\chi\circ G^+\, dd^c (G^- + \Omega) \wedge \om \\
& \leq & 
2 \int -\chi\circ G^+ \,[\Omega+dd^c G^+] \wedge \om +2 \int -\chi \circ G^-\, [\Omega+dd^c G^-] \wedge \om.\\
& = &
2 \int -\chi\circ G^+\, T^+ \wedge \om +2 \int -\chi\circ G^- \, T^- \wedge \om.
\end{eqnarray*}
We have just seen that the last two integrals are finite, so the proof is complete.
\end{proof}

Theorem \ref{thm:nospurs} is now an immediate consequence of Proposition \ref{pro:grad} and Remark \ref{rqe:nrgl1}.
\qed

\section{Examples}\label{sec:examples}

In this section we exhibit families of examples satisfying our energy conditions. 
Recall from Theorem 4.2 in \cite{ddg1} that the assumption $\lambda_1(f)>\lambda_2(f)$ implies that $X$ is either rational or of Kodaira dimension zero.

\subsection{Polynomial mappings of $\cd$} 
Suppose that $f:\C^2 \rightarrow \C^2$ is {\it polynomial} and, as always, has small topological degree. 
Recent  work of Favre and Jonsson \cite{fj} gives a smooth compactification $X=\C^2 \cup D_{\infty}$ of $\C^2$ 
and $k \in \N^*$ such that 
\begin{itemize}
\item[(P1)] the meromorphic extension of $f^k$ to $X$ is 1-stable;
\item[(P2)] $f^k$ contracts the divisor $D_{\infty}$  at infinity to a fixed point in $I^- \setminus I^+(f^k)$;
\item[(P3)] the relative potential $G^+$ for $T^+$ is continuous in $X \setminus I^+(f^k)$.
\end{itemize}
Using this information, we will show\footnote{After this article was accepted for publication, it was pointed out to the authors that the conclusions of Theorem \ref{thm:polynomial} (as well as results about polynomial maps in \cite{ddg1, ddg3}) hold not only for the iterate $f^k$ but in fact for the map $f$ itself.   The interested reader may consult \cite[\S 4.1.1]{ddg1} for more details about this.}:

\begin{thm}\label{thm:polynomial}
Let $f:\C^2 \rightarrow \C^2$ be a polynomial mapping with $\lambda_2(f)<\lambda_1(f)$, and let $X, k$ be as in (P1)-(P3).  Then $f^k:X\to X$ has finite dynamical energy.  More precisely $G^+\in \E^1(T^-,\om^+)$
and $G^-\in\E^1(T^+,\om^-)$.
\end{thm}

\begin{proof} 
To simplify notation, we replace $f^k$ by $f$.  We recall that $1$-stability implies that $I^+\cap I^-=\emptyset$.  Since $G^+$ is continuous in $X \setminus I^+$, it is in particular finite at all points in $I^-$, so $G^+ \in \E^1(T^-,\om^+)$ by Proposition \ref{pro:energycriter} and Corollary \ref{cor:e1}.

Though $G^-$ is less well-behaved, we can show it is finite at all points in $I^+$.
Since $f$ is polynomial on $\C^2$, we have $I^+\subset D_\infty$.  
The invariance $f(\C^2) \subset \C^2$ and the contraction property $f(D_{\infty} \setminus I^+)=q=f(q) \in I^-$ imply that $f^{-1}(I^+)\subset I^+$.  Thus
$$
G^-(p) =    \sum_{n\geq 0} \frac{1}{\l^n} (f^n)_* \Gamma^- (p) \geq
        \sum_{n\geq 0} \frac{1}{\l^n} (f^n)_* M > -\infty,
$$
where $M=\min_{p\in I^+} \Gamma^-(p)$ is finite because $I^+\cap I^- = \emptyset$ (see \cite[Lemma 3.2]{ddg1}). 
 From Corollary \ref{cor:e1} again, we see that $G^-\in \nabla^1(T^+,\omega^-)$.  Since $G^+\in L^1(T^-)$, we have $G^-\in L^1(T^+)$ by symmetry.  Hence in fact $G^-\in \E^1(T^+,\omega^-)$.
\end{proof}

It follows from Theorem \ref{thm:mixing} that $\mu_f=T^+ \wedge T^-$ is a well-defined, mixing invariant 
probability measure which does not charge pluripolar sets.  We can actually say more.

\begin{thm}
Let $f$ be as in Theorem \ref{thm:polynomial}, and $V \subset X$ be an algebraic curve.  
Then $\log {\rm dist}(\cdot,V) \in L^1(\mu_f)$.
\end{thm}

In particular $\log {\rm dist}(\cdot,{\mathcal C}_f) \in L^1(\mu_f)$, where
${\mathcal C}_f$ denotes the critical set of $f$.
This result will allow us in \cite{ddg3} to use Pesin's theory of non-uniformly 
hyperbolic dynamical systems and show the
existence of many saddle periodic points.

\begin{proof}
Assume without loss of generality that $V$ is irreducible.  We claim that $G^-_{|V} \not\equiv -\infty$ on $V$.  Granting this for the moment, let $0 \geq \f_V \in L^1(X)$ be a global potential for the current of integration along $V$.  Thus $dd^c \f_V=[V]-\Theta$, where $\Theta$ is smooth, and $\f_V \in L^1(\mu_f)$ if and only if $\log {\rm dist}(\cdot,V) \in L^1(\mu_f)$.  From our claim we have that $G^-\in L^1([V])$ and, by symmetry $\varphi\in L^1(T^-)$.  Hence the measure $T^- \wedge [V]$ is well-defined. 

Fix a constant $\gamma$ between $\sqrt{\lambda_2}$ and $\lambda_1$.  Then from \cite{fj}, we have $C_1>0$ such that 
$$
{\rm dist}(f^n x,I^+) \geq (C_1 {\rm dist}(x,I^+))^{\g^n},\quad\text{for all } x\in X\text{ and } n\in\N
$$
Since $\Gamma^+(x) \geq A \log {\rm dist}(x,I^+)-B$, we infer from \eqref{T+} that
$G^+(x) \geq A' \log {\rm dist}(x,I^+)-B'.$
Therefore repeated integration by parts gives (up to finite additive/multiplicative constants) 
\begin{eqnarray*}
\int-\varphi\, \mu_f & = & \int -G^+ \,T^- \wedge [V] \leq \int -\log \mathrm{dist}(\cdot,I^+)\, T^- \wedge [V] \\
& = &
\int_V -G^- \,dd^c\log \mathrm{dist}(\cdot,I^+)\wedge [V]
\leq \int_V -G^-\, \omega\wedge[V] + \sum_{x\in I^+\cap V} G^-(x) < \infty,
\end{eqnarray*}
since $G^-$ is finite on $I^+$.

It remains to verify that $G^-$ is not identically $-\infty$ on $V$.  Since $V$ must meet $D_\infty$ somewhere, it suffices to show that $G^-$ is finite on $V\cap D_\infty$.  This is the case if, for instance, $V\cap D_\infty\subset I^+$.  However, from $f_*G^-=\l(G^- - \Gamma^-)$, we see that we need only show that $G^-$ is finite on $f^{-n}(V)\cap D_\infty$ for some $n\in\N$.  If, for instance, $V\cap D_\infty$ does not contain the superattracting
point $q\in I^-$, then this is true for $n=1$ because $f^{-1}(D_\infty\setminus\{q\})\subset I^+$.  Finally, as the next lemma makes clear, even if $D_\infty\cap V$ does contain $q$, the same reasoning works for some larger value of $n$.
\end{proof}

\begin{lem}
There exists  an integer $n \in \N$ such that if $W$ is  any irreducible component 
 of $f^{-n}(V)$, not contained in $D_{\infty}$, then $q\notin W$.
\end{lem}

\begin{proof}
Suppose on the contrary that for each $n\in\N$, there exists an irreducible
component $W_n$ of $f^{n*} V$ such that
$q\in W_n$ but $W_n\not\subset D$.  Then since $q$ is superattracting for $f$, this
remains true at the local level. 
That is, there is a neighborhood $U\ni q$ and for every $n\in\N$ a local irreducible
component $W_n$ of the (local) pullback $(f|_U)^{n*}(V)$ such that $q\in W_n$ but
$W_n\not\subset D$.  

Moreover, by \cite{fj} we may choose a coordinate patch $U$ about $q=\mathbf{0}$ so that
$f:(U,\mathbf{0})\to (U,\mathbf{0})$ is a rigid holomorphic germ and there exists  a non-trivial
local divisor $D_{loc}\subset D$ satisfying $(f\rest{U})^* D_{loc} \geq \l
D_{loc}$.  The local topological degree of $f_{|U}$ is no larger than $\lambda_2$.  Hence
$$
\l^n \leq \pair{W_n}{f^{n*} D_{loc}}_q = \pair{f^n_* W_n}{D_{loc}}_q \leq
\lambda_2^n\pair{V}{D_{loc}}_q,
$$
where $\pair{\cdot}{\cdot}_q$ denotes local intersection of germs at $q$.
This contradicts $\lambda_2<\l$ for large $n$.
\end{proof}

\subsection{The secant method}
If $P:\C\to \C$ is a polynomial of degree at least two such that all roots of $P$ are simple, then we recall the secant method from \cite{ddg1}: given points $x,y\in\C$, one declares $f(x,y) = (y,z)$ where $z$ is chosen so that
$(0,z)$ lies on the line from $(x,P(x))$ to $(y,P(y))$.  This prescription defines a $1$-stable meromorphic map $f:\P^1\times\P^1\to\P^1\times\P^1$ with small topological degree.  

\begin{pro}
The secant map $f$ has finite dynamical energy.  More precisely, $G^+\in \E^1(T^-,\om^+)$ and $G^-\in\E^1(T^+,\om^-)$.
\end{pro}

\begin{proof}
We have that $I^- = I^-(f^n) = \{(z,z):P(z)=0\}$ consists of fixed points for every $n\in\N$ and that each of these is attracting for $f$.  It follows that $\dist(f^n(I^-),I^+), \dist(f^{-n}(I^+),I^-) \geq c$ for some $c>0$ and every $n\in\N$.  As with the case of $G^-$ for polynomial maps of $\C^2$ then, we infer that $G^+$ is finite on $I^-$ and vice versa.  Thus by Corollary \ref{cor:e1}, $G^+ \in \nabla^1(T^-,\omega^+)$ and $G^-\in\nabla^1(T^+,\omega^-)$.  

The classes of $T^+$ and $T^-$ are both K\"ahler, moreover, so from Theorem \ref{thm:nospurs} we see that 
$G^+\in\E^1(T^-,\omega^+)$ and $G^-\in\E^1(T^+,\omega^-)$.
\end{proof}

\subsection{Kodaira dimension zero}

When $\mathrm{kod}(X)=0$, we may assume after birational conjugation that $X$ is minimal
and $f$ is 1-stable (see \cite[Proposition 4.3]{ddg1}).

\begin{pro}
Assume $X$ is a minimal surface with $kod(X)=0$. 
Then $f$ satisfies condition $(E_{\chi})$ with $\chi(t)=t$.
Moreover $\log {\rm dist}(\cdot,I^+) \in L^1(\mu_f)$.
\end{pro}

\begin{proof}
We know from \cite{ddg1} that $f$ is non-ramified.  Therefore $I^-$ is empty and (see \cite[Proposition 4.10]{ddg1}) $G^-$ is continuous on $X$.  That is, the hypotheses of Corollary \ref{cor:e1} are satisfied by both $G^+$ and $G^-$. From Proposition \ref{pro:energycriter}, we further have $\log {\rm dist}(\cdot,I^+) \in L^1(\mu_f)$.
\end{proof}

\subsection{Irrational rotations}\label{subs:irrational}

The following examples originate in the work of Favre \cite{favre-per} (see also \cite{b}). 
Their common feature is the existence of a complex line where $f$ is conjugate to a rotation. 
Choosing the rotation angle properly allows us to produce functions $G^\pm$ that are 
very singular and therefore useful for testing the sharpness of the energy conditions. 

\begin{exa}
Given $a \in \C$, we consider the birational  transformation
$$
f:[x:y:t] \in \P^2 \mapsto [y^2:ay^2+t^2-xy:yt] \in \P^2,
$$
For (the all but countably many) parameters $a\in\C$ such that $f$ is $1$-stable on $\P^2$, 
we have $\lambda_1(f) = 2$ is the degree of the homogeneous polynomials defining $f$.  
Thus $f$ has small topological degree.  Moreover, since
$\dim H^{1,1}(\P^2) = 1$, it follows that we may take $\omega = \omega^+ = \omega^-$.  
So there are no spurious points in either $I^+$ or $I^-$.

One checks that $I^+ = \{[1:0:0]\}$, $I^- = \{[0:1:0]\}$, and that the line $L:=(t=0)$ joining 
these two points is $f$-invariant.  When $a\in\C\setminus [-2,2]$, we have 
$\overline{I_{\infty}^+} \cap \overline{I_{\infty}^-}=\emptyset$, where
$$
I_{\infty}^+=\bigcup_{n \geq 0} I^+(f^n)
\text{ and }
I_{\infty}^-=\bigcup_{n \geq 0} I^-(f^n)
$$
This implies that $f$ is $1$-stable and that $G^{\pm}$ is finite at points in $I^{\mp}$.  
>From Proposition \ref{pro:homograd}, Corollary \ref{cor:e1} and Theorem \ref{thm:nospurs} 
it follows that $f$ satisfies condition $(E_{\chi})$ for $\chi(t)=t$.

Now suppose that $a=2 \cos \pi\theta \in [-2,2]$.  In this case $f_L \eqdef f_{|L}$ is conjugate to a rotation of angle 
$2\pi\theta$, and $f$ is $1$-stable if and only if $\theta$ is irrational.  
Determining whether or not $f$ satisfies $(E_\chi)$ for any given $\chi$ is tricky.  
We will show for any given $p\in (0,1)$ that $G^\pm \in \E^p(T^\mp, \om^\pm)$ if $\theta$ is 
not too well approximated by rational numbers. In \cite{favre-per} it was precisely proved 
that $G^\pm \notin \E^1(T^\mp, \om^\pm)$ for certain values of $\theta$.

Observe that $f$ is conjugate to $f^{-1}$ by the involution $(x,y) \mapsto (y,x)$.  Hence by Corollary \ref{cor:ep} (also Proposition \ref{pro:homograd} and Theorem \ref{thm:nospurs}), we need only verify that
$\limsup_{t\to -\infty} |m^-(t)||t|^{q-1}\,dt <+\infty$, where $p$ is some number larger than $q$ and $m^-(t)$ is the mean value of $G^-$ on the sphere $\partial B_{[1:0:0]}(e^t)$.  By plurisubharmonicity, $m^-(t)$ is comparable to $\sup_{B(I^+,e^t)} G^-$, which is in turn bounded below by $m^-_L(t) \eqdef \sup_{B(I^+,e^t) \cap L} G^-$.  We therefore estimate the latter, fixing coordinates on $L\cong\P^1$ so that $f_L(x) = e^{2\pi i\theta} x$ for all $x\in\C\subset L$ and that $I^+$, $I^-$ become the points $1$ and $-1$, respectively.  Hence
$$
G^-(x) \simeq \sum_{n \geq 0} 2^{-n} \log |e^{2\pi i n\theta} x + 1|.
$$
Note for any $n\in\N$ and $t>0$ that $\sup_{|x-1| = e^{-t}} \log|e^{2\pi i n\theta} x + 1|$ is essentially achieved at $x = 1 + e^{-t}$.  Hence
$$
m_L^-(t) \approx \sum_{n\geq 0} 2^{-n} \max\{\log\epsilon(n),t\}
$$
where $\epsilon(n) \eqdef \min_{m\in\Z} |2n\theta - (2m+1)|$.  Now let $(\frac{2m_j+1}{2n_j})_{j\in\N} \subset\Q$ be the sequence uniquely determined by requiring $\gcd(2n_j,2m_j+1) = 1$, setting $n_0 = 1$ and choosing $n_j> n_{j-1}$ to be the smallest integer such that
$\epsilon(n_j) < \epsilon(n_{j-1})$.  From this it is entertaining to compute that
$$
\limsup |m_L^-(t)| |t|^{q-1}\,dt \approx \limsup 2^{-n_j} |\log \epsilon(n_j)|^q.
$$
In particular, $G^\pm\in \E^p(T^\mp, \om^\pm)$ if the right side is finite for some $q>p$.  We remark that finiteness holds for almost all $\theta\in\R$ and can be checked for any given irrational $\theta$ by examining its continued fraction expansion.
\vskip.2cm
\end{exa}

The next example is studied in  \cite{dg} where it is proved that $G^+ \in L^1(T^-)$. Hence 
we are in a situation where the alternative of Theorem \ref{thm:alternative} holds.

\begin{exa} \label{exa:3lines}
For parameters $a,b,c \in \C^*$, we consider the rational transformation
of the complex projective plane, $f=f_{abc}:\P^2 \rightarrow \P^2$, defined by
$$
f[x:y:z]=[bcx(-cx+acy+z):
acy(x-ay+abz):abz(bcx+y-bz)].
$$
The following facts can be verified by straightforward computation.
\begin{itemize}
\item $f_{abc}$ is birational with inverse $f^{-1} = f_{a^{-1}b^{-1}c^{-1}}$.
\item $I_f = \{[a:1:0],[0:b:1],[1:0:c]\}$.
\item $f$ preserves each of the lines $\{x=0\}$, $\{y=0\}$, $\{z=0\}$
  according to the formulas
$$
[x:1:0] \mapsto [-{bc}x:a:0],\;
[0:y:1]\mapsto \left[0:-{ac}y:b\right],\;
[1:0:z]\mapsto \left[c:0:-{ba}z\right]
$$
\end{itemize}
In particular, we have $I^\infty_f, I^\infty_{f^{-1}} \subset \{xyz=0\}$ for
all $a,b,c\in\C^*$.

Given $s>1$ and an irrational number $\theta\in\R$, let 
$f:\P^2\to\P^2$ be the birational map $f=f_{abc}$ with $a=i$, 
$b=-se^{2\pi i\theta}$, $c=i/s$.  One can then check 
(see \cite{dg}) that
\begin{itemize}
\item $f$ is $1$-stable on $X=\P^2$;
\item $T^+ \in L^1(T^-)$.
\end{itemize}

Thus the measure $\mu_f=T^+ \wedge T^-$ is a well defined probability measure.
It is further shown in \cite{dg} that $\mu_f$ does not charge curves and
is mixing. We can apply the alternative of Theorem \ref{thm:alternative},
reinforced by the ergodicity of $\mu_f$:
\begin{itemize}
\item either $\mu_f$ is supported on the pluripolar set
$\{G^+ + G^-=-\infty\}$, 
\item 
or $f$ has finite dynamical energy.
\end{itemize}
We expect that the latter always occurs.  When $\theta$ is not too close to rational numbers, this can be verified by 
arguing as in the previous example.
\end{exa}

\end{document}